\documentclass[11pt]{article}
\usepackage[a4paper,margin=0.82in]{geometry}
\usepackage{amsmath,amssymb,amsthm,mathtools,bm}
\usepackage{booktabs,graphicx,subcaption}
\usepackage{placeins}
\usepackage{float}
\usepackage{microtype}
\usepackage{enumitem}
\usepackage{multicol}
\usepackage{xcolor}
\usepackage{hyperref}
\usepackage{cleveref}
\hypersetup{colorlinks=true,linkcolor=blue!45!black,citecolor=blue!45!black,urlcolor=blue!45!black}
\setlist{leftmargin=*,nosep}
\newtheorem{definition}{Definition}[section]
\newtheorem{assumption}[definition]{Assumption}
\newtheorem{lemma}[definition]{Lemma}
\newtheorem{proposition}[definition]{Proposition}
\newtheorem{theorem}[definition]{Theorem}
\newtheorem{corollary}[definition]{Corollary}
\newtheorem{remark}[definition]{Remark}
\newcommand{\eps}{\varepsilon}
\newcommand{\dt}{\Delta t}
\newcommand{\Ps}{\Pi_s}
\newcommand{\Pf}{\Pi_f}

\newcommand{\Rc}{R_{\rm SFRC}}
\newcommand{\Tfour}{T_4}

\newcommand{\R}{\mathbb R}
\newcommand{\lemref}[1]{\hyperref[#1]{Lemma~\ref*{#1}}}
\newcommand{\Lemref}[1]{\hyperref[#1]{Lemma~\ref*{#1}}}
\newcommand{\thmref}[1]{\hyperref[#1]{Theorem~\ref*{#1}}}
\newcommand{\Thmref}[1]{\hyperref[#1]{Theorem~\ref*{#1}}}
\newcommand{\propref}[1]{\hyperref[#1]{Proposition~\ref*{#1}}}
\newcommand{\Propref}[1]{\hyperref[#1]{Proposition~\ref*{#1}}}
\newcommand{\corref}[1]{\hyperref[#1]{Corollary~\ref*{#1}}}
\newcommand{\Corref}[1]{\hyperref[#1]{Corollary~\ref*{#1}}}

\title{\bf Slow--Fast Response Correction for a Compact Fourth-Order IMEX Two-Derivative Method}
\author{Zhixin Huo\\
\small School of Mathematics and Information Science, Henan Polytechnic University\\
\small Jiaozuo 454003, Henan, China\\
\small Corresponding author: \texttt{zhixinhuo@hpu.edu.cn}}
\date{}

\begin{document}
\maketitle

\begin{abstract}
A compact two-stage fourth-order two-derivative IMEX method may retain its classical non-commuting fourth-order expansion and L-stable stiff damping while losing high-order accuracy when the relaxation time is comparable with the time step.  We ask whether uniform fourth-order accuracy can be recovered without redesigning the original two-stage, two-implicit-solve core.  For linear relaxation systems with a separated slow spectral subspace, we introduce a slow--fast response correction (SFRC): the exact slow projection is advanced by the unchanged compact IMEX step and projected back to the slow subspace, while the fast complement is advanced by its exact semigroup.  For fixed nonzero relaxation time, the correction is $O(\dt^5)$, so the base mixed expansion through degree four is unchanged, and the pure-fast response remains L-stable.  An abstract slow--fast argument reduces the full-state error to a slow-block defect and converts a uniform $O(\dt^5)$ local estimate into a uniform $O(\dt^4)$ global estimate.  For each Jin--Xin Fourier mode, the required local estimate is proved by an exact projected-amplification factorization and a uniform bound in $x=\eps/\dt$.  In particular, the unchanged compact core already has a uniform fifth-order one-step defect after restriction to the exact finite-$\eps$ slow eigenspace.  A two-dimensional three-variable model yields the same mechanism on fixed Fourier grids.  Symbolic verification and one- and two-dimensional experiments corroborate the analytical predictions.  SFRC is therefore a rigorous linear benchmark identifying the finite-$\eps$ slow--fast information sufficient to reconcile compactness, mixed fourth-order accuracy, stiff damping, and uniform fourth-order time accuracy.
\end{abstract}



\section{Introduction}
High-order time discretization of relaxation systems has to reconcile several requirements that are individually standard but collectively restrictive.  When the relaxation time $\eps$ is fixed and nonzero, the ordinary design order should be observed.  When $\eps\to0$, the method should reproduce the reduced equilibrium dynamics without resolving the fast time scale.  In the intermediate regime $\eps=O(\dt)$, the error constant should remain controlled rather than deteriorating.  At the same time, a practical IMEX method should damp stiff modes strongly and should retain the mixed explicit--implicit elementary differentials generated by a non-commuting split.  The distinction among classical order, asymptotic preservation, asymptotic accuracy, and uniform accuracy is by now well established for hyperbolic relaxation and kinetic equations \cite{CaflischJinRusso1997,BoscarinoRusso2009,HuShuBDF2021,HuShuRK2025,MaHuang2025}.

Uniform-accuracy mechanisms have been developed in several complementary time-discretization frameworks.  Hu and Shu established uniform-accuracy results for selected IMEX-BDF and IMEX-RK schemes under appropriate regularity and stage-structure hypotheses \cite{HuShuBDF2021,HuShuRK2025}, while Ma and Huang treated general linear hyperbolic relaxation systems under a structural-stability framework \cite{MaHuang2025}.  These analyses select or design schemes so that the required uniform order and stability conditions are satisfied.  The question here is different: we start from a \emph{fixed}, coefficient-rigid, two-stage two-derivative IMEX core whose classical fourth-order and stiff-damping properties are already prescribed, and ask whether uniform fourth order can be recovered without changing those stage coefficients.  A recent sequential two-stage fourth-order transport--relaxation construction with ADER trajectory derivatives \cite{HuoSequential2026} provides another route to high-order stiff accuracy, but it uses a different sequential temporal architecture.  The present work instead keeps the original compact IMEX core on the exact finite-$\eps$ slow subspace and isolates the additional slow--fast response information sufficient for UA4.  Accordingly, no claim of a first uniform-accuracy theory is made; the contribution is a compatibility and diagnostic result for this specific compact multiderivative topology.

The starting point of this work is the compact two-stage fourth-order two-derivative IMEX method introduced in \cite{HuoCompact2026}.  Its defining feature is that the split temporal derivatives are evaluated along the complete vector field,
\begin{equation}
 \dot F=F_U(F+G),\qquad \dot G=G_U(F+G).
 \label{eq:fullder}
\end{equation}
For a linear non-commuting split $F(U)=AU$, $G(U)=BU$, this gives
\[
 \dot F=A(A+B)U=A^2U+ABU,\qquad
 \dot G=B(A+B)U=BAU+B^2U,
\]
so the $AB$ and $BA$ interactions are present intrinsically rather than reconstructed through a large collection of additive Runge--Kutta coupling conditions.  The method uses one intermediate stage and two nontrivial implicit solves per step, is fourth order for fixed split operators, and its pure-implicit stability function is A-stable and L-stable with $O(|z|^{-2})$ stiff decay \cite{HuoCompact2026}.  These properties are the structural assets that the present paper is required to preserve.

A separate companion asymptotic analysis by the author shows that ordinary fourth order does not by itself guarantee fourth-order accuracy after the singular relaxation limit is taken \cite{HuoAA2026}.  For equilibrium-prepared linear Jin--Xin data, the unmodified strict limiting macro map has an $O(\dt^2)$ one-step defect and therefore only first-order fixed-time convergence.  That analysis also gives a coefficient-rigidity result: once the midpoint stage and two-solve topology are fixed, ordinary fourth-order non-commuting conditions determine the available final weights, so the strict-limit defect cannot be removed by a simple coefficient retuning.  An equilibrium-defect correction repairs the strict endpoint, but it does not repair the finite-$\dt/\eps$ fast response.  This motivates a stronger design question: can one modify the original method while retaining its compact temporal core, the base method's mixed expansion through order four, and the L-stability property, and at the same time obtain a genuinely uniform fourth-order estimate?  This is the core problem of the paper.  It is not an endpoint-consistency problem: the modification must remain invisible to the ordinary fourth-order non-commuting expansion at fixed $\eps$, must preserve strong damping of the fast block, and must control the entire distinguished family $\eps/\dt\in(0,\infty)$ with one error constant.

The answer given here is yes for a clearly delimited linear spectral setting.  We assume that the semidiscrete relaxation operator admits a separated slow spectral subspace and that the corresponding projector is available.  The corrected method first decomposes the old state into slow and fast spectral components.  The slow component is advanced with exactly the original two-stage two-derivative IMEX formula and then projected back to the slow subspace.  The fast component is advanced with its exact semigroup.  No coefficient in the original stage formula is changed and no additional implicit solve is introduced.  This statement concerns implicit solves only: SFRC additionally requires the exact slow projection and exact fast semigroup action, operations that are inexpensive for the small constant-coefficient Fourier blocks studied here but are not negligible in general.  This construction is called the \emph{slow--fast response correction} (SFRC).

The main mathematical point is not the definition itself but the proof that the desirable properties survive simultaneously.  For fixed $\eps>0$, the original base step and the exact flow agree through $O(\dt^4)$.  Since the exact flow is block diagonal with respect to the exact spectral projectors, replacing the base slow--fast leakage and the base fast block changes the full map only by $O(\dt^5)$.  Thus the $h$-Taylor operator coefficients through degree four are unchanged, and hence the evaluated non-commuting mixed expansion of the base method is preserved.  In the pure-fast scalar limit SFRC reduces to $e^z$, hence it is L-stable; the property is preserved, but the base rational stability function itself is not.  For uniform accuracy, the exact fast component cancels from the numerical error completely, and the problem reduces to a slow-block fourth-order estimate.

A uniform-accuracy claim requires the slow-block local estimate to be established with a constant independent of $\eps$, not merely postulated.  We close this point explicitly for the Fourier-mode Jin--Xin system.  We derive the slow-projected amplification factor $q(h,\eps,d)$ exactly and prove the factorization
\[
 q(h,\eps,d)-\Tfour(h\lambda_s)=h^5\mathcal C(h,\eps,d),
\]
where $\lambda_s$ is the exact slow eigenvalue.  Introducing $x=\eps/h$, using the characteristic relation $\lambda_s+cd=\eps(a^2d^2-\lambda_s^2)$, and exploiting that $d=i\kappa$ makes the implicit denominators have a strictly negative real part, we prove that $\mathcal C$ is uniformly bounded for all $x\in[0,\infty)$ under a non-coalescence condition on the slow and fast eigenvalues.  This yields a genuine $O(h^5)$ local bound independent of $\eps$.  A separate discrete stability estimate then converts the local estimate into a global $O(h^4)$ bound.  The proof covers arbitrary initial data because the fast component is advanced exactly, not merely damped.

The two-dimensional extension is also more than a numerical analogy.  For the three-variable periodic relaxation model used below, the variables $u$ and $p=d_xv+d_yq$ form a closed $2\times2$ relaxation block, while the remaining direction is fast.  The base stage equations intertwine with this reduction.  Consequently the same slow-amplification factorization applies after replacing $cd$ by $c_xd_x+c_yd_y$ and $a^2d^2$ by $a_x^2d_x^2+a_y^2d_y^2$.  An explicit full $3\times3$ slow-projector formula then gives the required uniform projector bound under a checkable non-coalescence condition, yielding a modewise fourth-order uniform result on any fixed Fourier grid satisfying that condition.

The contribution is therefore best viewed as one connected structural result rather than a collection of independent modifications.  The method keeps the original compact IMEX core on the slow subspace and supplies exactly the finite-$\eps$ slow--fast response information that the endpoint correction lacks.  The theory then proves three consequences of this construction: the correction is $O(h^5)$ at fixed $\eps$ and hence preserves the base non-commuting expansion through order four; the pure-fast block remains L-stable; and, once a uniform fifth-order slow local defect is established, the full-state global error is uniformly fourth order.  The nontrivial model-specific contribution is the proof of that uniform slow defect for Jin--Xin by an exact factorization and an $x=\eps/h$ bound.  An immediate consequence, stated explicitly below, is that the unchanged compact base core is uniformly fourth-order accurate in its one-step response after restriction and projection to the exact finite-$\eps$ slow eigenspace.  This pinpoints the structural mechanism behind the earlier order loss: the obstruction is the mismatch between strict-equilibrium information and the finite-$\eps$ slow--fast decomposition, not a hidden low-order defect of the slow-projected two-stage core.  The two-dimensional reduction and the reproducible symbolic and numerical experiments are included to show that the same mechanism survives beyond a single scalar amplification test.  Thus the central claim is a compatibility theorem for a linear spectral benchmark: compact two-stage structure, the original fourth-order mixed expansion, L-stability, strict-limit fourth order, and uniform-in-$\eps$ fourth-order time accuracy can coexist.

The remainder of the paper is organized as follows.  Section~2 recalls the compact fourth-order IMEX two-derivative base formula, its linear one-step matrix, and the precise operator-level meaning of preserving the base mixed expansion.  Section~3 defines SFRC by the exact slow--fast spectral splitting and proves the fixed-$\eps$ fifth-order perturbation property together with preservation of the L-stability property on the pure-fast test.  Section~4 isolates the general error mechanism and proves the abstract implication ``uniform fifth-order slow local defect plus uniformly bounded slow dynamics $\Rightarrow$ full-state uniform fourth-order convergence.''  Section~5 verifies the difficult local-defect hypothesis for each Jin--Xin Fourier mode by a cancellation-resistant slow eigenvalue, a uniformly bounded projector, an exact algebraic factorization, and a bound that is uniform for all $x=\eps/h\ge0$; it then isolates the resulting uniform slow-block accuracy of the unchanged compact core as a structural corollary and identifies the strict relaxation limit.  Section~6 derives the two-dimensional three-variable reduction and proves the corresponding uniform-in-$\eps$ temporal result on a fixed Fourier grid.  Section~7 presents four numerical examples, each with the full refinement history, figures, and interpretation: structural/L-stability diagnostics, a modal relaxation scan including distinguished paths, a one-dimensional periodic Jin--Xin problem, and a two-dimensional periodic relaxation problem.  Section~8 discusses the main contribution, the exact-projector/exact-semigroup limitation, and the additional analytical ingredients that would be required for a genuinely nonlinear UA4 extension.  Section~9 concludes, while the appendices provide an independently checkable algebraic certificate, the stiff-stability propagation argument, and complete reproducibility information.

\section{The compact fourth-order base method and the properties to preserve}
Consider the additive split problem
\begin{equation}
 U_t=F(U)+G(U).
 \label{eq:split}
\end{equation}
With the full-vector-field derivatives \eqref{eq:fullder}, the compact midpoint stage is
\begin{equation}
 U^*=U^n+\frac h2F(U^n)+\frac{h^2}{8}\dot F(U^n)
      +\frac h2G(U^*)-\frac{h^2}{8}\dot G(U^*),
 \label{eq:stage}
\end{equation}
and the final stage is
\begin{align}
 U^{n+1}_{\rm b}={}&U^n+hF(U^n)
 +\frac{h^2}{6}\bigl[\dot F(U^n)+2\dot F(U^*)\bigr]\notag\\
 &+hG(U^{n+1}_{\rm b})
 -\frac{h^2}{6}\bigl[\dot G(U^{n+1}_{\rm b})+2\dot G(U^*)\bigr],
 \label{eq:final}
\end{align}
where $h=\dt$.  The ordinary fourth-order consistency and stability analysis of \eqref{eq:stage}--\eqref{eq:final} is given in \cite{HuoCompact2026}.  We record the linear step matrix because it is used throughout the proofs.

For
\begin{equation}
 U_t=(A+B)U=:LU,
 \label{eq:linear}
\end{equation}
let $S(h;A,B)$ denote the one-step matrix.  Solving the two stage equations gives
\begin{align}
 P_*={}&\left(I-\frac h2B+\frac{h^2}{8}BL\right)^{-1}
 \left(I+\frac h2A+\frac{h^2}{8}AL\right),
 \label{eq:pstar}\\
 S(h;A,B)={}&\left(I-hB+\frac{h^2}{6}BL\right)^{-1}
 \left[I+hA+\frac{h^2}{6}AL+\frac{h^2}{3}(A-B)LP_*\right].
 \label{eq:S}
\end{align}
For every fixed pair of bounded matrices $A,B$,
\begin{equation}
 S(h;A,B)=e^{h(A+B)}+O(h^5),\qquad h\to0.
 \label{eq:fixedorder}
\end{equation}
The estimate in \eqref{eq:fixedorder} is a fixed-operator statement; the constant is not asserted to be uniform when $B=B_0/\eps$ becomes singular.

For $A=0$ and scalar $B=\lambda$, the base stability function is
\begin{equation}
 R_{\rm b}^{I}(z)=
 \frac{1-z/2-5z^2/24}
 {(1-z/2+z^2/8)(1-z+z^2/6)},\qquad z=h\lambda.
 \label{eq:RI}
\end{equation}
The original analysis proves A-stability and
\[
 R_{\rm b}^{I}(z)=-10z^{-2}+O(z^{-3}),\qquad |z|\to\infty,
\]
so the implicit substructure is L-stable \cite{HuoCompact2026}.

We next formalize what it means for a correction to preserve the mixed property of the base scheme.

\begin{definition}[Preservation of the base mixed expansion through order four]\label{def:mixed-preservation}
A corrected one-step map $R_c(h;A,B)$ is said to preserve the base method's mixed expansion through order four if, for every fixed non-commuting pair $A,B$,
\[
 R_c(h;A,B)-S(h;A,B)=O(h^5).
\]
Then the $h$-Taylor operator coefficients of the corrected and base maps agree through order four.  Since the base coefficients are the established non-commuting word polynomials in $A$ and $B$, the corrected map reproduces the same evaluated mixed expansion for every fixed pair.
\end{definition}
This definition does not redefine the mixed-compatibility concept of the base method; it gives an operator-level criterion for proving that the already established mixed structure is left unchanged through the design order.  It is deliberately stronger than checking a few commuting scalar order conditions and does not require assigning a separate free-algebra representation to the spectral projector.  Since the base method uses $A(A+B)$ and $B(A+B)$, the distinct words $AB$ and $BA$ occur already at the derivative level.  An $O(h^5)$ perturbation leaves every degree-four operator coefficient of the base expansion unchanged.

\begin{remark}[Logical role of the companion asymptotic analysis]\label{rem:companion-role}
The strict-limit defect, coefficient-rigidity result, and equilibrium-defect correction cited in the Introduction are motivation from the companion study \cite{HuoAA2026}; none is an assumption in the theorems below.  The analysis in the present paper starts from the base formulas \eqref{eq:stage}--\eqref{eq:S} and proves the SFRC properties independently in Sections~3--6.  In particular, the uniform slow-block estimate, the full-state convergence theorem, and the strict SFRC limit are all derived here.  Thus the main compatibility theorem is logically self-contained even if the companion results are viewed only as motivation.
\end{remark}

\section{Slow--fast response correction}
We specialize to a finite-dimensional linear relaxation family
\begin{equation}
 U_t=L_\eps U,\qquad L_\eps=A+\frac1\eps B_0,
 \label{eq:Lrelax}
\end{equation}
with $0<\eps\le\eps_0$.  The matrices may depend on a fixed spatial Fourier mode, but $h$ and $\eps$ are the temporal and relaxation parameters.

\begin{assumption}[Spectral splitting]\label{ass:spectral}
For every $0<\eps\le\eps_0$, there are complementary spectral projectors
\begin{equation}
 I=\Ps(\eps)+\Pf(\eps),\qquad
 [L_\eps,\Ps]=[L_\eps,\Pf]=0,
 \label{eq:proj}
\end{equation}
with uniformly bounded norms,
\begin{equation}
 \sup_{0<\eps\le\eps_0}
 \bigl(\|\Ps(\eps)\|+\|\Pf(\eps)\|\bigr)\le C_P.
 \label{eq:projbound}
\end{equation}
The slow generator $L_s=\Ps L_\eps\Ps$ satisfies
\begin{equation}
 \sup_{0<\eps\le\eps_0}\|L_s(\eps)\|\le M_L.
 \label{eq:slowbound}
\end{equation}
\end{assumption}
Let $S_\eps=S(h;A,B_0/\eps)$ be the original compact base step.

\begin{definition}[SFRC]\label{def:sfrc}
The slow--fast response corrected map is
\begin{equation}
 \Rc(h,\eps)=\Ps S_\eps\Ps+e^{hL_\eps}\Pf.
 \label{eq:SFRC}
\end{equation}
Operationally, the old state is first decomposed as $U^n=U_s^n+U_f^n$ with $U_s^n=\Ps U^n$ and $U_f^n=\Pf U^n$.  The slow part $U_s^n$ is advanced by the original two-stage formula \eqref{eq:stage}--\eqref{eq:final} and then projected with $\Ps$; the fast part is advanced by $e^{hL_\eps}$.  Thus no new implicit stage or implicit solve is introduced.  The comparison is only at the level of implicit solves: the exact spectral projection and fast semigroup action are additional operations, and their practical cost is problem dependent.
\end{definition}

The correction can also be written relative to a full base step.  Set $E_\eps=e^{hL_\eps}$ and $D_\eps=S_\eps-E_\eps$.  Since $E_\eps$ commutes with the spectral projectors,
\begin{equation}
 \Rc-S_\eps=\Ps D_\eps\Ps-D_\eps.
 \label{eq:corrcompact}
\end{equation}
This identity makes the fixed-parameter order preservation immediate.

\begin{proposition}[Fifth-order perturbation and preservation of the base mixed expansion]\label{prop:mixed-preservation}
Fix $\eps>0$.  Then
\begin{equation}
 \Rc(h,\eps)-S_\eps=O(h^5),\qquad h\to0.
 \label{eq:O5corr}
\end{equation}
Consequently SFRC has ordinary order four and preserves the base method's non-commuting mixed expansion through degree four.
\end{proposition}
\begin{proof}
For fixed $\eps$, the matrices $A$ and $B_0/\eps$ are fixed, hence \eqref{eq:fixedorder} gives $D_\eps=O(h^5)$.  The spectral projector is bounded for this fixed $\eps$.  Equation \eqref{eq:corrcompact} therefore yields
\[
 \|\Rc-S_\eps\|
 \le \|\Ps\|^2\|D_\eps\|+\|D_\eps\|=O(h^5).
\]
The Taylor expansions of $\Rc$ and $S_\eps$ consequently agree through $h^4$, equivalently
\[
 \partial_h^k\Rc(0,\eps)=\partial_h^k S_\eps(0),\qquad 0\le k\le4.
\]
Thus the $h$-Taylor operator expansion of the corrected map through degree four is identical to that already established for the base method.  In particular, the evaluated mixed combinations containing the distinct $AB$ and $BA$ contributions generated by \eqref{eq:fullder} are unchanged.
\end{proof}

\begin{remark}
The proposition is a fixed-$\eps$ statement, which is exactly the level at which ordinary order and the base method's non-commuting mixed expansion are compared.  Uniform accuracy requires a different estimate because the fixed-$\eps$ constant behind $O(h^5)$ may blow up as $\eps\to0$.  Closing that gap is the purpose of Sections 4--6.
\end{remark}

\begin{proposition}[Pure-fast L-stability]\label{prop:lstable}
For the scalar purely fast equation $u_t=\lambda u$, $\Re\lambda\le0$, the slow projector is zero and SFRC has stability function
\begin{equation}
 R_{\rm SFRC}^{I}(z)=e^z,\qquad z=h\lambda.
 \label{eq:expstab}
\end{equation}
Hence $|R_{\rm SFRC}^{I}(z)|\le1$ for $\Re z\le0$ and $R_{\rm SFRC}^{I}(z)\to0$ as $\Re z\to-\infty$; SFRC is L-stable.
\end{proposition}
\begin{proof}
For a purely fast scalar block, $\Ps=0$ and $\Pf=I$.  Substitution into \eqref{eq:SFRC} gives $\Rc=e^{h\lambda}$.  The stated bounds are the standard bounds for the exponential on the closed left half-plane.
\end{proof}

The corrected scheme therefore preserves the L-stability property, but it does not preserve the original rational stability function.  On the fast block the base $O(|z|^{-2})$ algebraic stiff decay is replaced by the exact exponential response.  Example~1 compares the two responses quantitatively.

\section{Abstract full-state uniform fourth-order theorem}
We separate the error mechanism that is common to all linear spectral blocks from the model-specific task of proving a uniform local defect.  Define the numerical and exact slow one-step operators
\[
 \mathcal M_s(h,\eps)=\Ps S_\eps\Ps,
 \qquad
 \mathcal E_s(h,\eps)=e^{hL_\eps}\Ps=e^{hL_s}\Ps.
\]
The notation is chosen deliberately: $M_L$ in \eqref{eq:slowbound} is a scalar bound for the slow generator, whereas $\mathcal M_s$ is the numerical slow map.

\begin{assumption}[Uniform slow local defect]\label{ass:uniformlocal}
There exist $h_0>0$ and $C_\ell>0$, independent of $\eps$, such that
\begin{equation}
 \|\mathcal M_s(h,\eps)-\mathcal E_s(h,\eps)\|\le C_\ell h^5,
 \qquad 0<h\le h_0,
 \quad 0<\eps\le\eps_0.
 \label{eq:uniformlocal}
\end{equation}
\end{assumption}

The proof of uniform accuracy has three logically separate steps.  First, the exact fast block is represented without numerical error.  Second, \eqref{eq:uniformlocal} provides a local defect whose constant does not deteriorate as $\eps/h$ varies.  Third, only the slow numerical map is iterated, and its generator is uniformly bounded, so the stability factor is $e^{CT}$ rather than $e^{CT/\eps}$.  The next lemma and theorem make these steps explicit.

\begin{lemma}[Uniform stability of the corrected slow block]\label{lem:slow-stability}
Under \eqref{eq:projbound}, \eqref{eq:slowbound}, and \eqref{eq:uniformlocal}, there is a constant $C_T$, independent of $h$ and $\eps$, such that
\begin{equation}
 \|\mathcal M_s(h,\eps)^n\|\le C_T,
 \qquad 0\le nh\le T.
 \label{eq:Msstable}
\end{equation}
\end{lemma}
\begin{proof}
Because $\|L_s\|\le M_L$, the exact slow evolution restricted to $\operatorname{Ran}\Ps$ satisfies
\[
 \|\mathcal E_s v\|=\|e^{hL_s}v\|\le e^{M_Lh}\|v\|,
 \qquad v\in\operatorname{Ran}\Ps.
\]
This restriction is essential: the full-space estimate would contain a factor $\|\Ps\|$ at every step.  The map $\mathcal M_s=\Ps S_\eps\Ps$ leaves $\operatorname{Ran}\Ps$ invariant.  Hence for $v$ in that range, Assumption~\ref{ass:uniformlocal} gives
\[
 \|\mathcal M_s v\|
 \le \|\mathcal E_s v\|+C_\ell h^5\|v\|
 \le \bigl(e^{M_Lh}+C_\ell h^5\bigr)\|v\|.
\]
Choose $C_1$ independent of $h$ and $\eps$ such that $e^{M_Lh}+C_\ell h^5\le e^{C_1h}$ for $0<h\le h_0$.  Iteration on the invariant slow range yields
\[
 \|\mathcal M_s^n v\|\le e^{C_1nh}\|v\|\le e^{C_1T}\|v\|.
\]
For a general vector $w$, $\mathcal M_s^nw=\mathcal M_s^n\Ps w$; using \eqref{eq:projbound} therefore gives
$\|\mathcal M_s^nw\|\le C_Pe^{C_1T}\|w\|$, proving \eqref{eq:Msstable}.
\end{proof}

\begin{theorem}[Full-state uniform fourth-order time accuracy]\label{thm:abstract-UA4}
Assume the spectral splitting \eqref{eq:proj}--\eqref{eq:slowbound} and the uniform slow local defect \eqref{eq:uniformlocal}.  Then for every $T>0$ there is a constant $C_T$, independent of $\eps$ and $h$, such that
\begin{equation}
 \max_{0\le nh\le T}
 \|\Rc(h,\eps)^nU^0-e^{nhL_\eps}U^0\|
 \le C_T h^4\|U^0\|.
 \label{eq:UA4abstract}
\end{equation}
The estimate holds for arbitrary initial data because the fast spectral component is advanced exactly.
\end{theorem}
\begin{proof}
Set $\mathcal E_f=e^{hL_\eps}\Pf$.  By the spectral commutation relations and $\Ps\Pf=\Pf\Ps=0$,
\[
 \Rc=\mathcal M_s+\mathcal E_f,
 \qquad
 \mathcal M_s\mathcal E_f=\mathcal E_f\mathcal M_s=0.
\]
Consequently
\begin{equation}
 \Rc^n=\mathcal M_s^n+e^{nhL_\eps}\Pf.
 \label{eq:Rpower}
\end{equation}
The exact flow decomposes as
\begin{equation}
 e^{nhL_\eps}=e^{nhL_s}\Ps+e^{nhL_\eps}\Pf.
 \label{eq:Epower}
\end{equation}
Subtracting \eqref{eq:Epower} from \eqref{eq:Rpower} gives the exact cancellation
\begin{equation}
 \Rc^n-e^{nhL_\eps}=\mathcal M_s^n-e^{nhL_s}\Ps.
 \label{eq:fastcancel}
\end{equation}
Thus there is no fast numerical error to propagate.

On the slow range, $\mathcal E_s^j=e^{jhL_s}\Ps$.  The exact telescoping identity is
\begin{equation}
 \mathcal M_s^n-\mathcal E_s^n
 =\sum_{j=0}^{n-1}\mathcal M_s^{n-1-j}
 (\mathcal M_s-\mathcal E_s)\mathcal E_s^j.
 \label{eq:telescoping}
\end{equation}
By \lemref{lem:slow-stability}, the first factor is bounded uniformly.  By \eqref{eq:slowbound} and \eqref{eq:projbound},
$\|\mathcal E_s^j\|\le C_Pe^{M_Ljh}\le C_Pe^{M_LT}$.  Assumption~\ref{ass:uniformlocal} therefore implies
\[
 \|\mathcal M_s^n-\mathcal E_s^n\|
 \le C_T\sum_{j=0}^{n-1}h^5
 \le C_Tnh^5
 \le C_T h^4,
\]
where the constant has been enlarged but remains independent of $h$ and $\eps$.  Combining this estimate with \eqref{eq:fastcancel} and applying the operator bound to $U^0$ proves \eqref{eq:UA4abstract}.
\end{proof}

\begin{remark}[What the abstract theorem does and does not prove]\label{rem:abstract-scope}
The theorem is a rigorous implication, but \eqref{eq:uniformlocal} is the model-specific burden.  The theorem does not claim that every compact IMEX method automatically satisfies that estimate, nor does it provide a nonlinear projector construction.  In the next section we verify \eqref{eq:uniformlocal} analytically for the linear Jin--Xin Fourier block.  Because the fast semigroup is exact in SFRC, no slow-manifold preparation assumption is needed in this linear benchmark; replacing the exact fast evolution by an approximation would require a separate initial-layer error analysis.
\end{remark}

\section{Rigorous verification for the linear Jin--Xin Fourier block}
Consider
\begin{equation}
 u_t+v_x=0,\qquad
 v_t+a^2u_x=\frac{cu-v}{\eps},
 \qquad a^2>c^2.
 \label{eq:JX}
\end{equation}
On a Fourier mode with derivative eigenvalue $d=i\kappa$, $\kappa\in\R$, the split matrices are
\begin{equation}
 A=\begin{pmatrix}0&-d\\-a^2d&0\end{pmatrix},\qquad
 B=\frac1\eps\begin{pmatrix}0&0\\c&-1\end{pmatrix},
 \label{eq:JXAB}
\end{equation}
and
\[
 L_\eps=A+B.
\]
The characteristic equation is
\begin{equation}
 \eps\lambda^2+\lambda+cd-\eps a^2d^2=0.
 \label{eq:char}
\end{equation}
Let
\begin{equation}
 \zeta(\eps,d)=\sqrt{1-4\eps cd+4\eps^2a^2d^2}
 \label{eq:zeta}
\end{equation}
with the square-root branch continuous from $\zeta(0,d)=1$.  The slow root can be evaluated without subtractive cancellation as
\begin{equation}
 \lambda_s(\eps,d)
 =\frac{2(\eps a^2d^2-cd)}{1+\zeta(\eps,d)},
 \qquad
 \lambda_f=-\eps^{-1}-\lambda_s.
 \label{eq:lams}
\end{equation}

\begin{assumption}[Non-coalescence on the resolved Fourier set]\label{ass:gap}
For a prescribed wavenumber bound $|\kappa|\le K$ and $0\le\eps\le\eps_0$, assume
\begin{equation}
 |\zeta(\eps,i\kappa)|\ge\gamma>0.
 \label{eq:gap}
\end{equation}
This excludes a slow--fast eigenvalue collision on the resolved parameter set.
\end{assumption}
For any fixed finite Fourier grid, \eqref{eq:gap} is a checkable finite-mode condition.  It is satisfied by all parameter sets used in Section 7.

\subsection{Uniformly bounded spectral projector}
Since
\[
 \lambda_s-\lambda_f=2\lambda_s+\eps^{-1}=\frac{\zeta}{\eps},
\]
the rank-one slow projector admits the scaled form
\begin{equation}
 \Ps
 =\frac{\eps L_\eps+(1+\eps\lambda_s)I}{\zeta}
 =\frac1\zeta
 \begin{pmatrix}
 1+\eps\lambda_s & -\eps d\\
 c-\eps a^2d & \eps\lambda_s
 \end{pmatrix}.
 \label{eq:PsJX}
\end{equation}

\begin{lemma}[Uniform spectral bounds]\label{lem:spectral-bounds}
Under \eqref{eq:gap}, for $|\kappa|\le K$ and $0\le\eps\le\eps_0$ there are constants $M_\lambda$ and $C_P$, independent of $\eps$, such that
\begin{equation}
 |\lambda_s(\eps,i\kappa)|\le M_\lambda,
 \qquad
 \|\Ps(\eps,i\kappa)\|+\|\Pf(\eps,i\kappa)\|\le C_P.
 \label{eq:spectralbounds}
\end{equation}
\end{lemma}
\begin{proof}
The chosen square-root branch has $\Re\zeta\ge0$, hence $|1+\zeta|\ge1$.  Formula \eqref{eq:lams} gives
\[
 |\lambda_s|
 \le2\bigl(\eps_0a^2K^2+|c|K\bigr)=:M_\lambda.
\]
Each numerator entry in \eqref{eq:PsJX} is then bounded uniformly on the compact parameter set.  Dividing by $|\zeta|\ge\gamma$ gives a uniform bound for $\Ps$ in any fixed matrix norm.  Finally $\Pf=I-\Ps$, so $\Pf$ is uniformly bounded as well.
\end{proof}

\subsection{Exact projected amplification factorization}
For $d\ne0$, a right slow eigenvector and a left slow eigenvector can be chosen as
\begin{equation}
 r_s=\begin{pmatrix}1\\-\lambda_s/d\end{pmatrix},\qquad
 \ell_s^T=\begin{pmatrix}-(\lambda_s+\eps^{-1})/d&1\end{pmatrix}.
 \label{eq:eigvecs}
\end{equation}
Then $\Ps=r_s\ell_s^T/(\ell_s^Tr_s)$.  Since the slow range is one-dimensional,
\begin{equation}
 \Ps S_\eps\Ps=q(h,\eps,d)\Ps,
 \qquad
 q=\frac{\ell_s^TS_\eps r_s}{\ell_s^Tr_s}.
 \label{eq:qdef}
\end{equation}
Define
\[
 \Tfour(z)=1+z+\frac{z^2}{2}+\frac{z^3}{6}+\frac{z^4}{24}.
\]
The following exact identity is the central algebraic step.

\begin{lemma}[Exact algebraic factorization with an independently checkable certificate]\label{lem:factorization}
Let $\lambda=\lambda_s(\eps,d)$ satisfy \eqref{eq:char}.  Then
\begin{equation}
 q(h,\eps,d)-\Tfour(h\lambda)
 =-\frac{h^5\lambda^3(cd+\lambda)\,N}
 {24(1+2\eps\lambda)D_8D_6},
 \label{eq:factor}
\end{equation}
where
\begin{align}
 D_8={}&cd\eps h^2-8\eps^2-4\eps h-h^2,
 \label{eq:D8}\\
 D_6={}&cd\eps h^2-6\eps^2-6\eps h-h^2,
 \label{eq:D6}
\end{align}
and
\begin{align}
 N={}&cd\eps^3h^3\lambda^2-2cd\eps^3h^2\lambda-12cd\eps^3h
      -cd\eps^2h^3\lambda-6cd\eps^2h^2-cd\eps h^3 \notag\\
 &-14\eps^4h\lambda^2+16\eps^4\lambda-4\eps^3h^2\lambda^2
   +8\eps^3h\lambda+8\eps^2h^2\lambda+12\eps^2h \notag\\
 &+2\eps h^3\lambda+6\eps h^2+h^3.
 \label{eq:N}
\end{align}
\end{lemma}
\begin{proof}
Write $\lambda=\lambda_s$ and introduce $w=Br_s$.  The characteristic relation
\begin{equation}
 a^2d^2=\lambda^2+\frac{\lambda+cd}{\eps}
 \label{eq:achar}
\end{equation}
shows that the two-dimensional space spanned by $r_s$ and $w$ is invariant under $A$, $B$, and $L_\eps$.  In the ordered basis $V=(r_s,w)$ the three operators are
\begin{equation}
 \widehat A=\begin{pmatrix}\lambda&-(cd+\lambda)/\eps\\-1&-\lambda\end{pmatrix},\qquad
 \widehat B=\begin{pmatrix}0&0\\1&-1/\eps\end{pmatrix},\qquad
 \widehat L=\widehat A+\widehat B.
 \label{eq:reduced-basis}
\end{equation}
Moreover, projection with the left slow eigenvector becomes the scalar functional
\begin{equation}
 \omega^T:=\frac{\ell_s^T V}{\ell_s^Tr_s}
 =\begin{pmatrix}1&-\dfrac{cd+\lambda}{1+2\eps\lambda}\end{pmatrix}.
 \label{eq:omega}
\end{equation}
Consequently $q=\omega^T\widehat S e_1$, where $\widehat S=S(h;\widehat A,\widehat B)$ is obtained from exactly the same two stage solves as \eqref{eq:pstar}--\eqref{eq:S}.  The implicit matrices are lower triangular in this basis and satisfy
\begin{equation}
 \det\widehat M_*= -\frac{D_8}{8\eps^2},\qquad
 \det\widehat M_{n+1}= -\frac{D_6}{6\eps^2}.
 \label{eq:det-certificate}
\end{equation}
Thus, after the scalar projection \eqref{eq:omega}, the reduced rational expression for $q$ has denominator exactly $24(1+2\eps\lambda)D_8D_6$.

It remains only to identify the numerator.  Clearing this denominator, define the polynomial certificate
\begin{equation}
 \mathcal R:=24(1+2\eps\lambda)D_8D_6\bigl(q-\Tfour(h\lambda)\bigr)
 +h^5\lambda^3(cd+\lambda)N.
 \label{eq:certificate}
\end{equation}
Substituting \eqref{eq:reduced-basis} into the two stage equations, evaluating the two displayed $2\times2$ triangular inverses, and using \eqref{eq:achar} reduces the assertion to the finite polynomial identity $\mathcal R\equiv0$.  Appendix~A records those intermediate matrices explicitly and shows how their determinants produce $D_8D_6$; after the common denominator is cleared, direct coefficient collection cancels every monomial and leaves the zero polynomial.  Thus the proof consists of the displayed analytical reduction to a finite polynomial identity; no floating-point fitting, asymptotic truncation, or numerically inferred cancellation is used.  Equation \eqref{eq:certificate} is precisely \eqref{eq:factor}.  For $d=0$, $\lambda=0$ and the identity follows by continuity.
\end{proof}

\subsection{Uniform bound across the distinguished scale}
The factor $h^5$ in \eqref{eq:factor} is not yet enough.  We must show that its coefficient does not grow when $\eps/h$ varies from zero to infinity.

\begin{lemma}[Uniform fifth-order slow local defect]\label{lem:uniform-local}
Under \eqref{eq:gap}, for $|\kappa|\le K$, $0<\eps\le\eps_0$, and $0<h\le h_0$, there is a constant $C$ independent of $h$ and $\eps$ such that
\begin{equation}
 |q(h,\eps,i\kappa)-e^{h\lambda_s(\eps,i\kappa)}|
 \le Ch^5.
 \label{eq:quniform}
\end{equation}
Consequently
\begin{equation}
 \|\Ps S_\eps\Ps-e^{hL_\eps}\Ps\|\le C h^5
 \label{eq:localJX}
\end{equation}
with a possibly larger constant independent of $h$ and $\eps$.
\end{lemma}
\begin{proof}
Set
\begin{equation}
 x=\frac\eps h\ge0,
 \qquad d=i\kappa.
 \label{eq:x}
\end{equation}
The characteristic equation gives the exact identity
\begin{equation}
 cd+\lambda=\eps(a^2d^2-\lambda^2)=xh(a^2d^2-\lambda^2).
 \label{eq:cdlambda}
\end{equation}
Write
\begin{equation}
 D_8=h^2\widehat D_8,
 \qquad
 D_6=h^2\widehat D_6,
 \end{equation}
where
\begin{align}
 \widehat D_8&=cdxh-(8x^2+4x+1),\\
 \widehat D_6&=cdxh-(6x^2+6x+1).
 \label{eq:Dhat}
\end{align}
Because $d=i\kappa$ and $c\in\R$, the first term in each expression is purely imaginary.  Hence
\begin{align}
 |\widehat D_8|&\ge8x^2+4x+1\ge1+x^2,\notag\\
 |\widehat D_6|&\ge6x^2+6x+1\ge1+x^2.
 \label{eq:Dlower}
\end{align}
Moreover $1+2\eps\lambda=\zeta$, so \eqref{eq:gap} gives
\begin{equation}
 |1+2\eps\lambda|\ge\gamma.
 \label{eq:gap2}
\end{equation}

Substitute $\eps=xh$ into \eqref{eq:N} and factor $h^3$:
\begin{equation}
 N=h^3\widehat N(x,h,\lambda),
 \label{eq:Nhatdef}
\end{equation}
with
\begin{align}
 \widehat N={}&cdh^3\lambda^2x^3-2cdh^2\lambda x^3-cdh^2\lambda x^2
 -12cdh x^3-6cdh x^2-cdhx \notag\\
 &-14h^2\lambda^2x^4-4h^2\lambda^2x^3
 +16h\lambda x^4+8h\lambda x^3+8h\lambda x^2+2h\lambda x
 +12x^2+6x+1.
 \label{eq:Nhat}
\end{align}
The spectral bound $|\lambda|\le M_\lambda$ and the constraint $hx=\eps\le\eps_0$ imply
\begin{equation}
 |\widehat N(x,h,\lambda)|\le C_N(1+x^3).
 \label{eq:Nbound}
\end{equation}
To see this term by term, powers such as $h x^4$ are written as $(hx)x^3\le\eps_0x^3$, $h^2x^4=(hx)^2x^2\le\eps_0^2x^2$, and $h^3x^3=(hx)^3\le\eps_0^3$; all remaining terms are lower powers of $x$ with bounded coefficients.

Using \eqref{eq:cdlambda}, \eqref{eq:Nhatdef}, and \eqref{eq:Dhat} in \eqref{eq:factor} yields
\begin{equation}
 q-\Tfour(h\lambda)=h^5\mathcal C(h,x,\lambda),
 \label{eq:Cdef}
\end{equation}
where
\begin{equation}
 \mathcal C=
 -\frac{\lambda^3x(a^2d^2-\lambda^2)\widehat N}
 {24(1+2\eps\lambda)\widehat D_8\widehat D_6}.
 \label{eq:Ccoef}
\end{equation}
The numerator apart from $\widehat N$ is uniformly bounded except for the explicit factor $x$.  Equations \eqref{eq:Dlower}--\eqref{eq:Nbound} give
\[
 |\mathcal C|
 \le C\frac{x(1+x^3)}{(1+x^2)^2}
 \le C\frac{x+x^4}{1+x^4}
 \le C.
\]
This proves
\begin{equation}
 |q-\Tfour(h\lambda)|\le Ch^5
 \label{eq:qT4}
\end{equation}
uniformly in $h$ and $\eps$.

Finally, the standard exponential remainder and $|\lambda|\le M_\lambda$ give
\[
 |e^{h\lambda}-\Tfour(h\lambda)|
 \le \frac{e^{h_0M_\lambda}M_\lambda^5}{5!}h^5.
\]
Combining this bound with \eqref{eq:qT4} proves \eqref{eq:quniform}.  Since
\[
 \Ps S_\eps\Ps=q\Ps,
 \qquad
 e^{hL_\eps}\Ps=e^{h\lambda}\Ps,
\]
\Lemref{lem:spectral-bounds} supplies a uniform bound on $\|\Ps\|$, and \eqref{eq:localJX} follows.
\end{proof}

This proof explicitly resolves all three coupled ranges $\eps\ll h$, $\eps=O(h)$, and $\eps\gg h$ in a single estimate.  The ratio $x=\eps/h$ is unbounded as $h\to0$ with fixed $\eps$, but the denominator grows like $x^4$, exactly balancing the largest allowed numerator growth.

\begin{corollary}[Uniform slow-block accuracy of the unchanged compact core]\label{cor:slow-core}
Under the hypotheses of \lemref{lem:uniform-local},
\begin{equation}
 \sup_{\substack{0<\eps\le\eps_0\\0<h\le h_0}}
 h^{-5}\bigl\|\Ps S_\eps\Ps-e^{hL_\eps}\Ps\bigr\|<\infty.
 \label{eq:slow-core-uniform}
\end{equation}
Hence, for data in the exact finite-$\eps$ slow eigenspace, applying the original compact two-stage IMEX core and projecting the completed step back to that same slow eigenspace has a one-step defect $O(h^5)$ with a constant uniform in $\eps/h$.
\end{corollary}
\begin{proof}
Equation \eqref{eq:slow-core-uniform} is exactly the operator estimate \eqref{eq:localJX} of \lemref{lem:uniform-local}.  If $U\in\operatorname{Ran}\Ps$, then $\Ps U=U$, and therefore
\[
 \Ps S_\eps\Ps U-e^{hL_\eps}U
 =\bigl(\Ps S_\eps\Ps-e^{hL_\eps}\Ps\bigr)U,
\]
so the same uniform $O(h^5)$ bound applies to the slow-eigenspace one-step response.
\end{proof}

\begin{remark}[What the corollary identifies structurally]\label{rem:slow-core-meaning}
\Corref{cor:slow-core} does \emph{not} assert that the unprojected base method is uniformly fourth-order for arbitrary data.  It isolates a narrower and more informative fact: the original compact temporal core, when restricted to the correct finite-$\eps$ slow eigenspace, already has the required uniform fourth-order local response.  Thus the first-order strict-equilibrium map found for the unmodified method is not evidence of an intrinsic low-order defect in the slow-projected two-stage core.  The missing ingredient is the finite-$\eps$ slow--fast decomposition and the control of leakage/fast response outside that subspace; SFRC supplies exactly that information in the present linear benchmark.
\end{remark}

\begin{theorem}[Uniform fourth order for a Jin--Xin Fourier mode]\label{thm:JX-UA4}
Let $d=i\kappa$ with $|\kappa|\le K$, and assume \eqref{eq:gap}.  Then SFRC satisfies
\begin{equation}
 \max_{0\le nh\le T}
 \|\Rc(h,\eps)^nU^0-e^{nhL_\eps}U^0\|
 \le C_T h^4\|U^0\|,
 \label{eq:JXUA}
\end{equation}
for all $0<\eps\le\eps_0$ and sufficiently small $h$, with $C_T$ independent of $\eps$.  The result holds for arbitrary modal initial data.
\end{theorem}
\begin{proof}
\Lemref{lem:spectral-bounds} verifies the uniformly bounded projectors and slow generator required by \thmref{thm:abstract-UA4}; \lemref{lem:uniform-local} verifies its uniform fifth-order slow local defect.  Hence \thmref{thm:abstract-UA4} applies directly.  For completeness, in the present rank-one setting one can see the global estimate even more explicitly.  The corrected and exact one-step maps are
\[
 \Rc=q\Ps+e^{h\lambda_f}\Pf,
 \qquad
 e^{hL_\eps}=e^{h\lambda_s}\Ps+e^{h\lambda_f}\Pf.
\]
Thus
\[
 \Rc^n-e^{nhL_\eps}
 =(q^n-e^{nh\lambda_s})\Ps.
\]
From \eqref{eq:quniform}, $|q|\le e^{hM_\lambda}+Ch^5\le e^{C_1h}$ for small $h$.  The scalar telescoping identity gives
\[
 |q^n-e^{nh\lambda_s}|
 \le n\max(|q|,|e^{h\lambda_s}|)^{n-1}|q-e^{h\lambda_s}|
 \le C_Tnh^5\le C_Th^4.
\]
Multiplication by the uniformly bounded projector proves \eqref{eq:JXUA}.
\end{proof}

\begin{remark}[Exact content of the Jin--Xin result]\label{rem:JX-scope}
The uniformity in \thmref{thm:JX-UA4} is with respect to $0<\eps\le\eps_0$ and the coupled ratio $\eps/h$ for Fourier modes in a fixed bounded set $|\kappa|\le K$.  The constant may depend on $K$, $a$, $c$, $T$, and the non-coalescence constant $\gamma$.  No estimate uniform in an increasing spatial cutoff is asserted, and no nonlinear relaxation theorem is inferred from the modal calculation.
\end{remark}

\begin{corollary}[Fixed Fourier grid]\label{cor:JX-grid}
For a Fourier spectral discretization of \eqref{eq:JX} with a fixed finite set of wavenumbers $|\kappa|\le K$, if \eqref{eq:gap} holds for every resolved mode, then the SFRC time discretization is fourth-order accurate uniformly for $0<\eps\le\eps_0$.  The constant may depend on the fixed Fourier cutoff (and hence on the fixed spatial grid) but not on $\eps$; no uniform-in-$K$ PDE estimate is claimed.
\end{corollary}
\begin{proof}
Apply \thmref{thm:JX-UA4} to each resolved Fourier mode and take the maximum of the finitely many modewise constants.  Parseval's identity transfers the estimate to the discrete $L^2$ norm.
\end{proof}

\subsection{Strict relaxation limit}
The slow eigenvalue and projector have the limits
\begin{equation}
 \lambda_s\to-cd,
 \qquad
 \Ps\to
 \begin{pmatrix}1&0\\c&0\end{pmatrix}
 =:P_0,
 \qquad \eps\to0.
 \label{eq:strictproj}
\end{equation}
Equation \eqref{eq:qT4} and the factorization show more precisely that, for fixed $h>0$,
\begin{equation}
 q(h,\eps,d)\longrightarrow\Tfour(-cdh).
 \label{eq:qlimit}
\end{equation}
At the same time $\Re\lambda_f\sim-1/\eps$, so $e^{h\lambda_f}\to0$.  Therefore
\begin{equation}
 \Rc(h,\eps)\longrightarrow \Tfour(-cdh)P_0.
 \label{eq:strictR}
\end{equation}
The strict limiting slow map is the fourth-order Taylor method for $u_t=-cd\,u$.  Thus SFRC is AA4 at the strict endpoint.  Notice that this conclusion follows from the same finite-$\eps$ slow response used in the uniform proof; no separate equilibrium-only endpoint switch is required.

\section{Two-dimensional modewise extension}
Consider the periodic three-variable relaxation system
\begin{equation}
 \begin{aligned}
 u_t+v_x+q_y&=0,\\
 v_t+a_x^2u_x&=\frac{c_xu-v}{\eps},\\
 q_t+a_y^2u_y&=\frac{c_yu-q}{\eps}.
 \end{aligned}
 \label{eq:2D}
\end{equation}
For one Fourier mode $(d_x,d_y)=(i\kappa_x,i\kappa_y)$, define
\begin{equation}
 s=c_xd_x+c_yd_y,
 \qquad
 r=a_x^2d_x^2+a_y^2d_y^2.
 \label{eq:sr}
\end{equation}
Let $p=d_xv+d_yq$.  Applying the row map
\[
 \mathcal T U=\begin{pmatrix}u\\p\end{pmatrix}
\]
to the split matrices gives closed reduced matrices
\begin{equation}
 \bar A=\begin{pmatrix}0&-1\\-r&0\end{pmatrix},\qquad
 \bar B=\frac1\eps\begin{pmatrix}0&0\\s&-1\end{pmatrix},
 \label{eq:ABbar}
\end{equation}
with the intertwining identities
\begin{equation}
 \mathcal T A=\bar A\mathcal T,
 \qquad
 \mathcal T B=\bar B\mathcal T.
 \label{eq:intertwineAB}
\end{equation}
The slow eigenvalue therefore solves
\begin{equation}
 \eps\lambda^2+\lambda+s-\eps r=0.
 \label{eq:2Dchar}
\end{equation}
Introduce the branch
\begin{equation}
 \zeta_2(\eps,d_x,d_y)=\sqrt{1-4\eps s+4\eps^2r},
 \qquad \Re\zeta_2\ge0,
 \label{eq:zeta2}
\end{equation}
continued from $\zeta_2(0)=1$, and write
\begin{equation}
 \lambda_s=\frac{2(\eps r-s)}{1+\zeta_2},
 \qquad
 \theta:=1+\eps\lambda_s=\frac{1+\zeta_2}{2}.
 \label{eq:2Dslowroot}
\end{equation}
For a prescribed finite Fourier grid we impose the directly checkable non-coalescence condition
\begin{equation}
 |\zeta_2(\eps,i\kappa_x,i\kappa_y)|\ge\gamma_2>0,
 \qquad 0\le\eps\le\eps_0,
 \label{eq:2Dgap}
\end{equation}
for every resolved mode.  This is the two-dimensional analogue of \eqref{eq:gap}.

\begin{lemma}[Uniform full slow projector for the two-dimensional block]\label{lem:2Dprojector}
On a fixed Fourier grid satisfying \eqref{eq:2Dgap}, the full $3\times3$ slow spectral projector is uniformly bounded for $0<\eps\le\eps_0$, and the corresponding slow generator is uniformly bounded.
\end{lemma}
\begin{proof}
Let $L_\eps=A+B$ denote the full $3\times3$ Fourier matrix.  Direct substitution into $L_\eps r_s=\lambda_s r_s$ and $\ell_s^TL_\eps=\lambda_s\ell_s^T$ gives the right and left slow eigenvectors
\[
 r_s=\begin{pmatrix}
 1\\[2pt]
 (c_x-\eps a_x^2d_x)/\theta\\[2pt]
 (c_y-\eps a_y^2d_y)/\theta
 \end{pmatrix},
 \qquad
 \ell_s^T=\begin{pmatrix}
 1 & -\eps d_x/\theta & -\eps d_y/\theta
 \end{pmatrix}.
\]
Using \eqref{eq:2Dchar},
\[
 \ell_s^Tr_s
 =1-\frac{\eps(s-\eps r)}{\theta^2}
 =\frac{1+2\eps\lambda_s}{1+\eps\lambda_s}
 =\frac{\zeta_2}{\theta}.
\]
Hence the exact rank-one projector is
\begin{equation}
 \Ps^{(2D)}=\frac{\theta}{\zeta_2}\,r_s\ell_s^T.
 \label{eq:2Dprojector}
\end{equation}
Because $\Re\zeta_2\ge0$, $|1+\zeta_2|\ge1$, so $|\theta|\ge1/2$.  On a fixed Fourier grid, $|s|$ and $|r|$ are bounded, and \eqref{eq:2Dslowroot} therefore gives a bound $|\lambda_s|\le M_{\lambda,2}$ independent of $\eps$.  Thus $|\theta|$ is also bounded from above, all entries of $r_s$ and $\ell_s$ are uniformly bounded, and \eqref{eq:2Dgap} bounds the remaining factor $1/|\zeta_2|$.  Consequently $\|\Ps^{(2D)}\|\le C_{P,2}$ uniformly in $\eps$, and the complementary projector is bounded because $\Pf^{(2D)}=I-\Ps^{(2D)}$.  Finally,
\[
 \Ps^{(2D)}L_\eps\Ps^{(2D)}=\lambda_s\Ps^{(2D)},
\]
so the slow generator is uniformly bounded as well.
\end{proof}

\begin{lemma}[Intertwining of the compact base step]\label{lem:intertwine}
Let $S(h;A,B)$ and $S(h;\bar A,\bar B)$ be defined by \eqref{eq:pstar}--\eqref{eq:S}.  If \eqref{eq:intertwineAB} holds, then
\begin{equation}
 \mathcal T S(h;A,B)=S(h;\bar A,\bar B)\mathcal T.
 \label{eq:intertwineS}
\end{equation}
\end{lemma}
\begin{proof}
Set $L=A+B$ and $\bar L=\bar A+\bar B$.  Equation \eqref{eq:intertwineAB} implies $\mathcal TL=\bar L\mathcal T$, and hence it also intertwines every product appearing in the stage matrices, for example $\mathcal TBL=\bar B\bar L\mathcal T$.  If $M$ and $\bar M$ are corresponding invertible stage matrices with $\mathcal TM=\bar M\mathcal T$, then multiplication by inverses gives $\mathcal TM^{-1}=\bar M^{-1}\mathcal T$.  Applying this observation first to the midpoint solve \eqref{eq:pstar} and then to the final solve \eqref{eq:S} proves \eqref{eq:intertwineS}.
\end{proof}

The reduced pair \eqref{eq:ABbar} is algebraically the same as the Jin--Xin $2\times2$ block with $cd$ replaced by $s$ and $a^2d^2$ replaced by $r$.  Because $s$ is purely imaginary for real coefficients and Fourier wavenumbers, the denominator estimates \eqref{eq:Dlower} remain valid verbatim.

\begin{theorem}[Modewise two-dimensional uniform fourth order]\label{thm:2D-UA4}
Fix a finite Fourier grid for \eqref{eq:2D} and assume \eqref{eq:2Dgap} for every resolved mode.  Then the Fourier-semidiscrete SFRC solution is fourth-order accurate in time uniformly in $0<\eps\le\eps_0$ on that fixed grid.  The theorem is a time-discretization result at fixed spatial resolution; no constant uniform in the Fourier cutoff is asserted.
\end{theorem}
\begin{proof}
Use the full slow right eigenvector $r_s$ from \lemref{lem:2Dprojector}.  Its first component equals one, hence $\mathcal T r_s\neq0$.  The intertwining identity gives
\[
 \bar L(\mathcal T r_s)=\mathcal T L_\eps r_s=\lambda_s\mathcal T r_s,
\]
so $\mathcal T r_s$ is the nonzero reduced slow right eigenvector.  Let $\bar\ell_s^T$ be a left slow eigenvector of the reduced matrix $\bar L=\bar A+\bar B$.  By \eqref{eq:intertwineAB}, $\ell_s^T:=\bar\ell_s^T\mathcal T$ is a full left slow eigenvector.  Under \eqref{eq:2Dgap} the reduced slow root is simple, so the left--right pairing $\bar\ell_s^T\mathcal T r_s$ is nonzero.  Therefore the scalar projected base amplification is
\[
 q_{2D}=\frac{\ell_s^TSr_s}{\ell_s^Tr_s}
 =\frac{\bar\ell_s^T\mathcal T Sr_s}{\bar\ell_s^T\mathcal T r_s}
 =\frac{\bar\ell_s^T\bar S\,\mathcal T r_s}{\bar\ell_s^T\mathcal T r_s},
\]
where \eqref{eq:intertwineS} was used in the last step.  Thus $q_{2D}$ is exactly the projected amplification of the reduced $2\times2$ block \eqref{eq:ABbar}.  The factorization of \lemref{lem:factorization} therefore applies with $cd\mapsto s$ and $a^2d^2\mapsto r$.  Since $s$ is purely imaginary, the same $x=\eps/h$ estimate proves a uniform $O(h^5)$ slow local defect.  Lemma~\ref{lem:2Dprojector} supplies the full-projector and slow-generator bounds required by \thmref{thm:abstract-UA4}; consequently the abstract theorem gives the stated full-state uniform fourth-order estimate.
\end{proof}

The only model-specific spectral restriction in \eqref{eq:2Dgap} is explicit and checkable.  A zero of $\zeta_2$ would merge the two eigenvalues of the coupled $(u,p)$ block and make the rank-one slow projector ill-conditioned; this restriction should not be hidden inside an unspecified constant.

\section{Numerical verification}
The numerical experiments are organized around the claims proved above.  They do not prove the theorems; rather, each experiment tests a distinct quantitative consequence and is designed to reveal an algebraic, implementation, or asymptotic failure if the corresponding theoretical mechanism were incorrect.  Table~\ref{tab:evidence-map} records this claim--evidence alignment.  Every example reports the complete four-level refinement history used to quote a rate.

\begin{table}[!htbp]
\centering
\footnotesize
\setlength{\tabcolsep}{4pt}
\caption{Relationship between the main theoretical claims and the numerical diagnostics.  The analytical column, not the numerical column, is the logical basis for each theorem.}
\label{tab:evidence-map}
\begin{tabular}{p{.22\textwidth}p{.27\textwidth}p{.38\textwidth}}
\toprule
Claim & Analytical basis & Numerical diagnostic\\
\midrule
Preservation of the base fourth-order mixed expansion & \propref{prop:mixed-preservation} & Fixed-$\eps$ correction has fifth-order one-step size (Table~\ref{tab:corr}, Fig.~\ref{fig:corr}).\\
L-stability of the corrected fast response & \propref{prop:lstable} & Base rational and exact exponential fast responses are compared in Table~\ref{tab:stability-samples} and Fig.~\ref{fig:stability}.\\
Uniform $O(h^5)$ slow-block defect of the unchanged core & \lemref{lem:uniform-local}, \corref{cor:slow-core} & Representative scaled defects are listed in Table~\ref{tab:local-samples}, while Fig.~\ref{fig:local} scans twelve decades of $\eps$.\\
Full-state uniform fourth-order time accuracy & \thmref{thm:JX-UA4}, \thmref{thm:2D-UA4} & Complete max-over-$\eps$ refinement histories, three distinguished ratios $\eps/h$, and 1D/2D fixed-grid fields, including two-dimensional unprepared data (Tables~\ref{tab:modal-eq}--\ref{tab:pde2d-unprepared}).\\
\bottomrule
\end{tabular}
\end{table}

The non-coalescence conditions used in the fixed-grid theorems are also checked numerically on the modes appearing in the experiments.  Table~\ref{tab:assumption-diagnostics} reports the smallest sampled scaled discriminant factor $|\zeta|$ (or $|\zeta_2|$ in 2D) and the largest sampled projector norm.  The 1D test has a minimum sampled factor about $9.10\times10^{-1}$ and a maximum projector norm about $1.51$; for the active 2D modes the corresponding values are about $4.48\times10^{-1}$ and $2.96$.  Thus both test sets remain separated from the coalescence excluded by the theory.  These diagnostics verify that the test cases stay inside the theorem's stated spectral regime; they are not substitutes for the analytical bounds.

\begin{table}[!htbp]
\centering
\small
\caption{Non-coalescence and projector diagnostics on the modes and relaxation ranges used in the numerical examples.  The factor $|\zeta|$ is the scaled slow--fast separation entering the projector formulas, not the unscaled eigenvalue distance.}
\label{tab:assumption-diagnostics}
\begin{tabular}{lrr}
\toprule
Test set & Min. $|\zeta|$ & Max. projector norm\\
\midrule
1D Jin--Xin, $|k|\le3$ & 9.097e-01 & 1.514\\
2D active modes & 4.484e-01 & 2.963\\
\bottomrule
\end{tabular}
\end{table}

The 1D Jin--Xin projector and slow amplification use cancellation-resistant closed forms, and the identity in \lemref{lem:factorization} is checked independently in exact symbolic arithmetic.  High-precision arithmetic is used only for the local-defect diagnostic at extremely small $\eps$.  The spatial examples use Fourier spectral differentiation to isolate the temporal mechanism.

\subsection{Example 1: structural preservation, local defect, implementation cross-check, and stiff response}
We begin with one Jin--Xin Fourier mode, $a=1.5$, $c=0.7$, $d=i$.  The experiment has four separate purposes.  First, at fixed $\eps=0.2$ we measure the full operator correction $\|R_{\rm SFRC}-R_{\rm base}\|_2$.  Second, we evaluate the scaled slow defect $h^{-5}\|\Ps S_\eps\Ps-e^{hL_\eps}\Ps\|_F$ with 70-digit arithmetic.  Third, on a moderate $(h,\eps)$ grid we compare the cancellation-resistant factorized implementation with the literal definition $\Ps S_\eps\Ps+e^{hL_\eps}\Pf$.  Fourth, we compare the pure-fast scalar responses of the base method and SFRC.

\begin{table}[!htbp]
\centering
\small
\setlength{\tabcolsep}{5pt}
\caption{Example 1(a): complete fixed-$\eps$ refinement history for the size of the SFRC correction, with $\eps=0.2$.}
\label{tab:corr}
\begin{tabular}{rrr}
\toprule
$h$ & $\|R_{\rm SFRC}-R_{\rm base}\|_2$ & Rate\\
\midrule
0.04000 & $1.6309\times10^{-6}$ & --\\
0.02000 & $6.0965\times10^{-8}$ & 4.742\\
0.01000 & $2.0889\times10^{-9}$ & 4.867\\
0.00500 & $6.8399\times10^{-11}$ & 4.933\\
\bottomrule
\end{tabular}
\end{table}

\begin{figure}[!htbp]
\centering
\begin{subfigure}{.48\textwidth}
\centering\includegraphics[page=1,width=\textwidth]{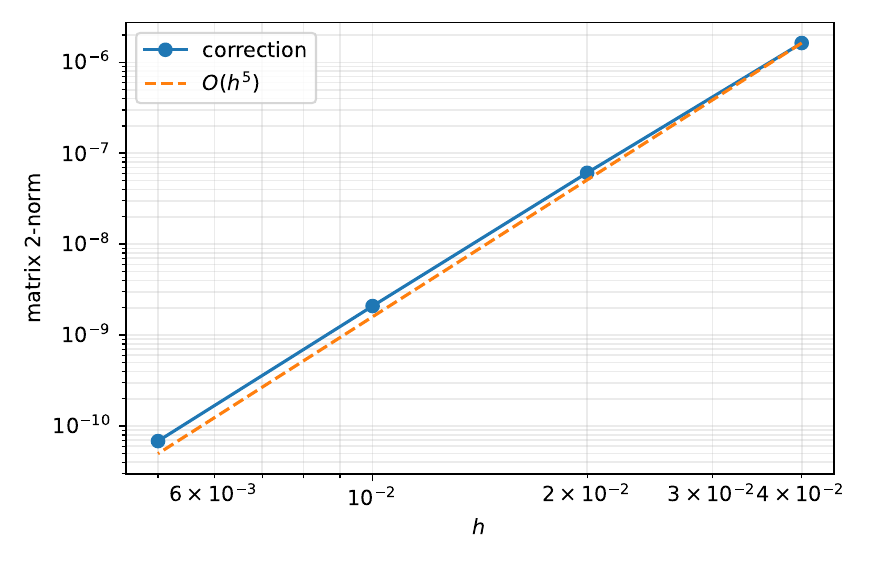}
\caption{Fixed-$\eps$ correction size.}
\label{fig:corr}
\end{subfigure}\hfill
\begin{subfigure}{.48\textwidth}
\centering\includegraphics[page=2,width=\textwidth]{figures.pdf}
\caption{Scaled slow defect for $10^{-12}\le\eps\le1$.}
\label{fig:local}
\end{subfigure}
\caption{Example 1: correction order and uniform local-defect diagnostics.  Panel (b) uses high precision so cancellation in matrices containing $1/\eps$ does not obscure the asymptotic coefficient.}
\end{figure}

The rates in Table~\ref{tab:corr} move monotonically toward five, which is the quantitative consequence of \propref{prop:mixed-preservation}: the correction is invisible to all terms through degree four at fixed $\eps$.  The local-defect plot tests a different statement.  Representative values are given in Table~\ref{tab:local-samples}; for $\eps=10^{-12}$ the scaled defect is about $1.71\times10^{-3}$ for all three steps, while for $\eps=1$ it stays near $7.1\times10^{-2}$.  The absence of growth as $\eps/h$ passes through the stiff and nonstiff regimes is consistent with the uniform coefficient bound proved in \lemref{lem:uniform-local} and directly visualizes the slow-block statement of \corref{cor:slow-core}.  It is important that the plotted quantity uses the unchanged base step $S_\eps$ between exact slow projections; the figure therefore tests the paper's central structural observation rather than only the final SFRC map.

\begin{table}[!htbp]
\centering
\small
\caption{Example 1(b): representative values of $h^{-5}\|\Ps S_\eps\Ps-e^{hL_\eps}\Ps\|_F$.}
\label{tab:local-samples}
\begin{tabular}{rrrr}
\toprule
$\varepsilon$ & $h=0.04$ & $h=0.02$ & $h=0.01$\\
\midrule
1e-12 & 1.7096e-03 & 1.7096e-03 & 1.7096e-03\\
1e-08 & 1.7096e-03 & 1.7096e-03 & 1.7097e-03\\
1e-04 & 1.7856e-03 & 1.8601e-03 & 2.0048e-03\\
1e+00 & 7.0309e-02 & 7.1095e-02 & 7.1499e-02\\
\bottomrule
\end{tabular}
\end{table}

\begin{table}[!htbp]
\centering
\small
\caption{Example 1(c): direct-vs-factorized implementation cross-check on $h\in\{0.04,0.02,0.01\}$ and $\eps\in\{10^{-3},10^{-2},10^{-1},1\}$.  The direct implementation forms $\Ps S_\eps\Ps+e^{hL_\eps}\Pf$ from the stage matrices, while the factorized implementation uses the closed scalar response of \lemref{lem:factorization}.}
\label{tab:implementation-crosscheck}
\begin{tabular}{lr}
\toprule
Cross-check grid & Maximum $\|R_{\rm factored}-R_{\rm direct}\|_2$\\
\midrule
$3\times4$ $(h,\varepsilon)$ pairs & $2.5123\times10^{-16}$\\
\bottomrule
\end{tabular}
\end{table}

The cross-check in Table~\ref{tab:implementation-crosscheck} is an implementation audit rather than a convergence test.  It verifies on an independent evaluation path that the factorized formula used for cancellation resistance reproduces the literal SFRC matrix to near floating-point precision on a moderate parameter grid.  This reduces the possibility that the observed uniform convergence is an artifact of coding the same algebraic formula into both the method and the reference diagnostic.

\begin{figure}[!htbp]
\centering
\includegraphics[page=3,width=.64\textwidth]{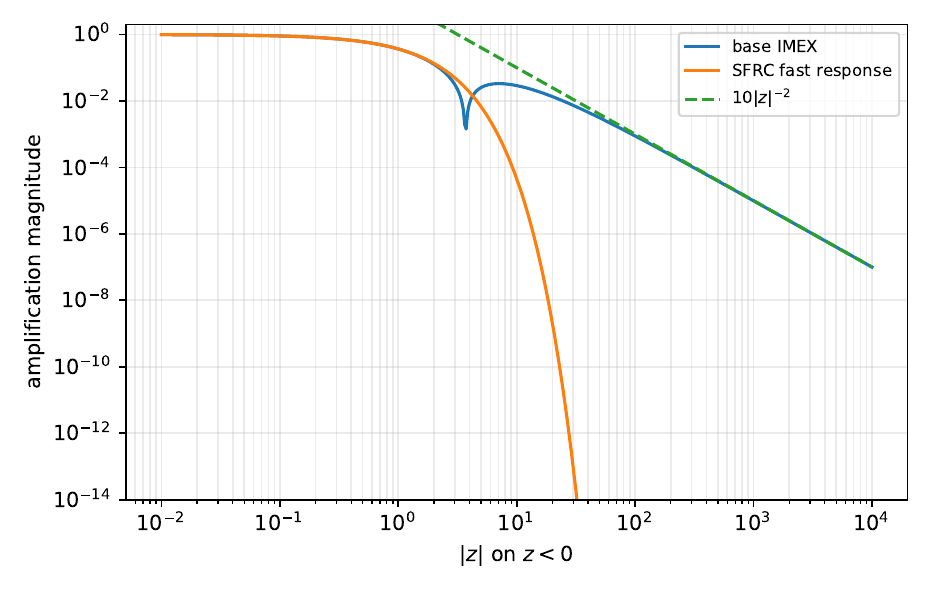}
\caption{Example 1(d): pure-fast scalar response.  The base method is already L-stable with $O(|z|^{-2})$ decay; SFRC has the same L-stability property while replacing the rational fast response by the exact exponential.}
\label{fig:stability}
\end{figure}

\begin{table}[!htbp]
\centering
\small
\caption{Example 1(d): representative amplification magnitudes on the negative real axis.}
\label{tab:stability-samples}
\begin{tabular}{rrr}
\toprule
$z$ & $|R_{\rm base}^{I}(z)|$ & $|R_{\rm SFRC}^{I}(z)|$\\
\midrule
-1 & $3.6686\times10^{-1}$ & $3.6788\times10^{-1}$\\
-10 & $2.8981\times10^{-2}$ & $4.5400\times10^{-5}$\\
-100 & $8.8373\times10^{-4}$ & $3.7201\times10^{-44}$\\
-1000 & $9.8768\times10^{-6}$ & $0$\\
\bottomrule
\end{tabular}
\end{table}

Table~\ref{tab:stability-samples} makes the distinction between ``preserving L-stability'' and ``preserving the same stability function'' explicit.  At $z=-1$ the two responses are of comparable size, whereas at $z=-10$ and beyond the exact exponential is already many orders of magnitude smaller.  The result therefore corroborates the stronger stiff damping predicted by \eqref{eq:expstab}; the analytical proof of L-stability remains \propref{prop:lstable}.

\FloatBarrier
\subsection{Example 2: modal uniform-accuracy scan and distinguished scales}
We next test the full-state global statement.  Set $T=0.2$ and scan $10^{-6}\le\eps\le1$ on a 43-point logarithmic grid for each
\[
 h\in\{0.04,0.02,0.01,0.005\}.
\]
Three methods are compared: the original compact base method, the strict-endpoint equilibrium-defect correction (EDC), and SFRC.  We use both limiting-equilibrium data $U^0=(1,c)^T$ and unprepared data $U^0=(1,0)^T$.  For each pair $(h,\eps)$ the error is maximized over all numerical times $nh\le T$; the convergence tables then maximize once more over the complete relaxation scan.

\begin{table}[!htbp]
\centering
\scriptsize
\setlength{\tabcolsep}{3pt}
\caption{Example 2(a): complete modal refinement history for equilibrium initial data, after maximizing the error over $10^{-6}\le\eps\le1$.}
\label{tab:modal-eq}
\begin{tabular}{rrrrrrr}
\toprule
$h$ & base error & Rate & EDC error & Rate & SFRC error & Rate\\
\midrule
0.04000 & $5.7270\times10^{-4}$ & -- & $2.6577\times10^{-3}$ & -- & $4.1101\times10^{-8}$ & --\\
0.02000 & $2.3981\times10^{-4}$ & 1.256 & $1.3705\times10^{-3}$ & 0.956 & $2.5843\times10^{-9}$ & 3.991\\
0.01000 & $1.1837\times10^{-4}$ & 1.019 & $6.9821\times10^{-4}$ & 0.973 & $1.6201\times10^{-10}$ & 3.996\\
0.00500 & $5.8638\times10^{-5}$ & 1.013 & $3.5301\times10^{-4}$ & 0.984 & $1.0141\times10^{-11}$ & 3.998\\
\bottomrule
\end{tabular}
\end{table}

\begin{table}[!htbp]
\centering
\scriptsize
\setlength{\tabcolsep}{3pt}
\caption{Example 2(b): complete modal refinement history for unprepared initial data.}
\label{tab:modal-unp}
\begin{tabular}{rrrrrrr}
\toprule
$h$ & base error & Rate & EDC error & Rate & SFRC error & Rate\\
\midrule
0.04000 & $2.2778\times10^{2}$ & -- & $2.2820\times10^{2}$ & -- & $2.9741\times10^{-8}$ & --\\
0.02000 & $5.6944\times10^{1}$ & 2.000 & $5.7004\times10^{1}$ & 2.001 & $1.8767\times10^{-9}$ & 3.986\\
0.01000 & $1.4232\times10^{1}$ & 2.000 & $1.4240\times10^{1}$ & 2.001 & $1.1789\times10^{-10}$ & 3.993\\
0.00500 & $3.5560\times10^{0}$ & 2.001 & $3.5570\times10^{0}$ & 2.001 & $7.3898\times10^{-12}$ & 3.996\\
\bottomrule
\end{tabular}
\end{table}

\begin{figure}[!htbp]
\centering
\includegraphics[page=4,width=.92\textwidth]{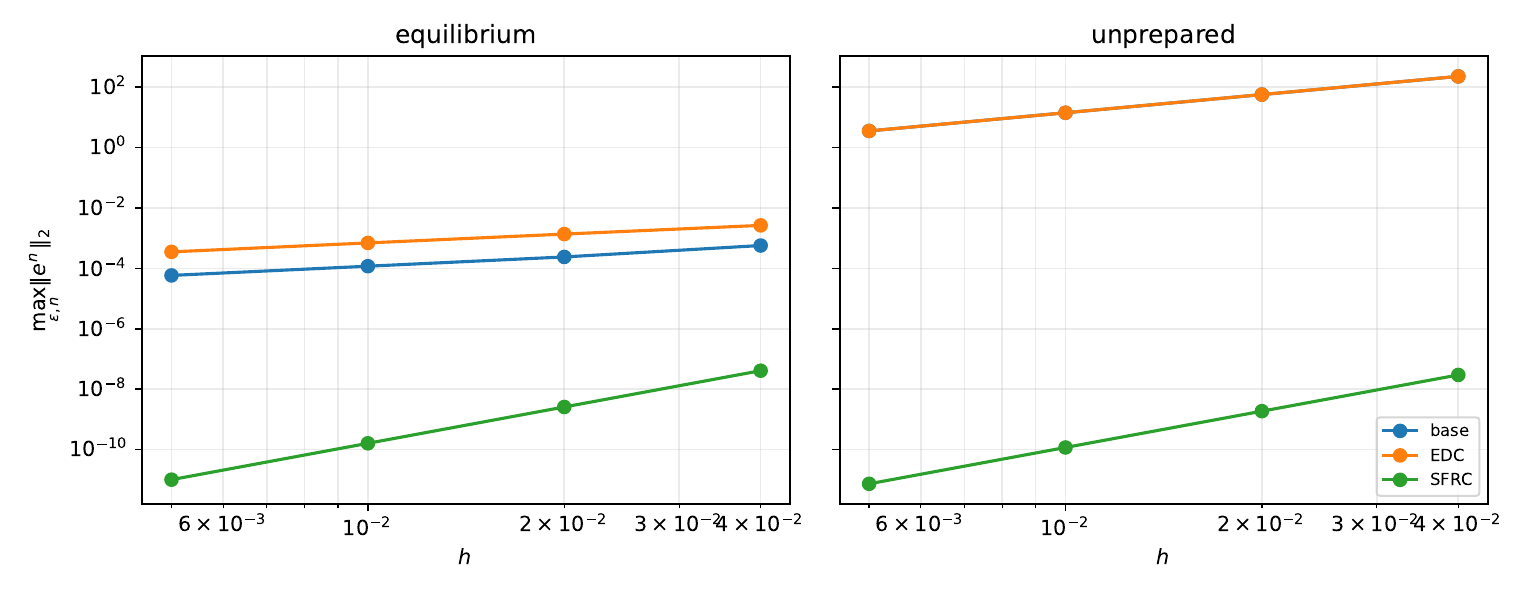}
\caption{Example 2: modal error maximized over the relaxation scan.  The base method and EDC retain an intermediate-scale barrier, while SFRC follows fourth-order convergence for both equilibrium and unprepared data.}
\label{fig:modal}
\end{figure}

For equilibrium data, the SFRC rate progresses from $3.991$ to $3.998$, while the base and EDC rates approach one.  For unprepared data, the contrast is even sharper: the base and EDC errors are dominated by the unresolved fast component and exhibit roughly second-order decay over this finite scan, whereas SFRC remains near fourth order because the fast spectral component is advanced exactly.  This is precisely the mechanism behind the preparation-free statement in \thmref{thm:JX-UA4}; the numerical data illustrate it but do not replace the block-diagonal error identity used in the proof.

Figure~\ref{fig:heat} displays the SFRC error over the entire $(h,\eps)$ grid rather than only its maximum.  There is no diagonal ridge near $\eps\approx h$, so the fourth-order behavior is not obtained by avoiding the distinguished scale.

\begin{figure}[!htbp]
\centering
\includegraphics[page=5,width=.92\textwidth]{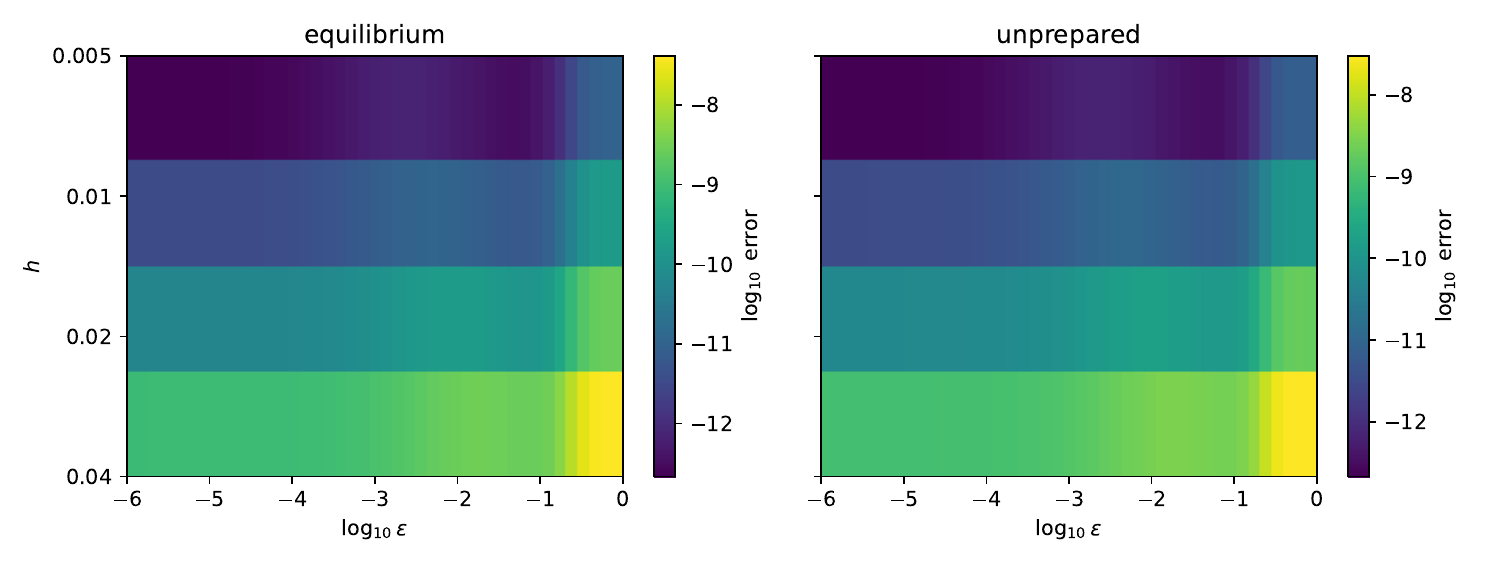}
\caption{Example 2: SFRC modal error as a function of $h$ and $\eps$.  The left and right panels correspond to equilibrium and unprepared data.}
\label{fig:heat}
\end{figure}

The $x=\eps/h$ proof in Section~5 is probed directly by three fixed ratios $x=0.1,1,10$.  Figure~\ref{fig:paths} and Tables~\ref{tab:paths-eq}--\ref{tab:paths-unp} report every refinement level.  The $x=0.1$ and $x=1$ sequences settle essentially exactly at order four.  The $x=10$ sequence shows pre-asymptotic rates above four because the leading fourth-order coefficient is small on this particular path; the theorem asserts a uniform lower guarantee of fourth order, not a universal superconvergence result.

\begin{figure}[!htbp]
\centering
\includegraphics[page=6,width=.88\textwidth]{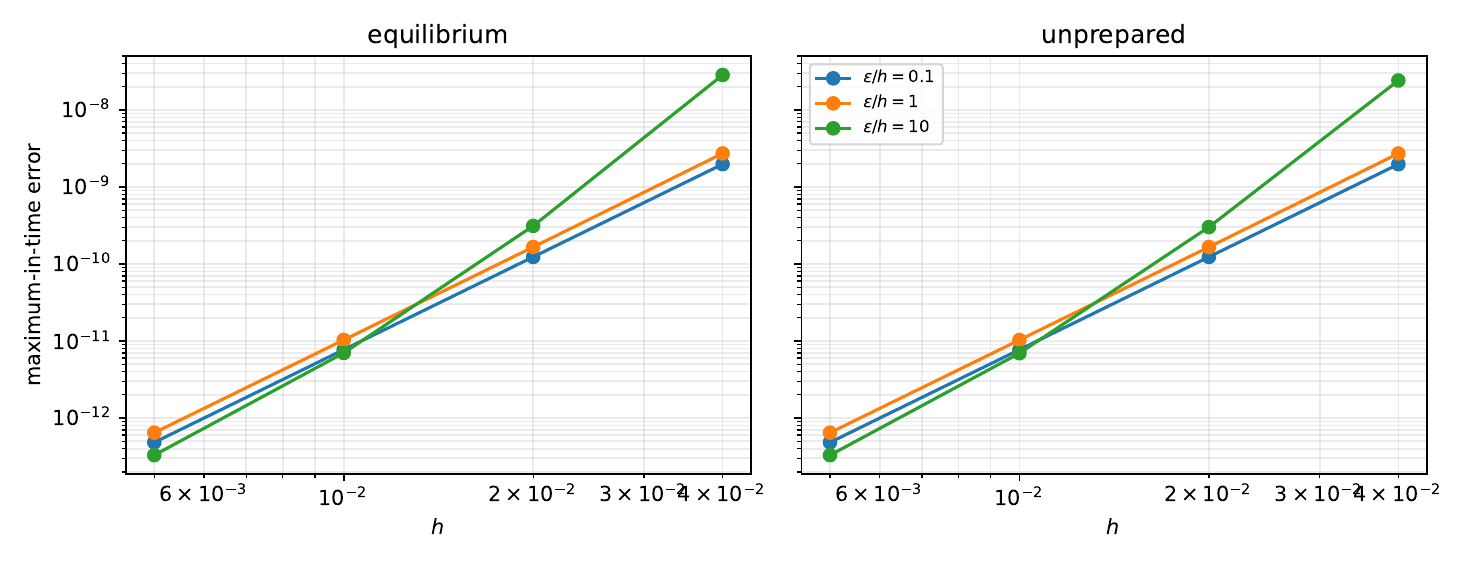}
\caption{Example 2: SFRC convergence along $\eps/h=0.1$, $1$, and $10$.}
\label{fig:paths}
\end{figure}

\begin{table}[!htbp]
\centering
\footnotesize
\caption{Example 2(c): complete distinguished-path history for equilibrium initial data.}
\label{tab:paths-eq}
\begin{tabular}{rrrrr}
\toprule
$x=\varepsilon/h$ & $h$ & Error & Rate & $\varepsilon$\\
\midrule
0.1 & 0.04000 & $1.9744\times10^{-9}$ & -- & 4.000e-03\\
0.1 & 0.02000 & $1.2343\times10^{-10}$ & 4.000 & 2.000e-03\\
0.1 & 0.01000 & $7.7160\times10^{-12}$ & 4.000 & 1.000e-03\\
0.1 & 0.00500 & $4.8225\times10^{-13}$ & 4.000 & 5.000e-04\\
\addlinespace[2pt]
1 & 0.04000 & $2.7229\times10^{-9}$ & -- & 4.000e-02\\
1 & 0.02000 & $1.6583\times10^{-10}$ & 4.037 & 2.000e-02\\
1 & 0.01000 & $1.0315\times10^{-11}$ & 4.007 & 1.000e-02\\
1 & 0.00500 & $6.4435\times10^{-13}$ & 4.001 & 5.000e-03\\
\addlinespace[2pt]
10 & 0.04000 & $2.8394\times10^{-8}$ & -- & 4.000e-01\\
10 & 0.02000 & $3.1188\times10^{-10}$ & 6.508 & 2.000e-01\\
10 & 0.01000 & $6.9927\times10^{-12}$ & 5.479 & 1.000e-01\\
10 & 0.00500 & $3.2982\times10^{-13}$ & 4.406 & 5.000e-02\\
\bottomrule
\end{tabular}
\end{table}

\begin{table}[!htbp]
\centering
\footnotesize
\caption{Example 2(d): complete distinguished-path history for unprepared initial data.}
\label{tab:paths-unp}
\begin{tabular}{rrrrr}
\toprule
$x=\varepsilon/h$ & $h$ & Error & Rate & $\varepsilon$\\
\midrule
0.1 & 0.04000 & $1.9743\times10^{-9}$ & -- & 4.000e-03\\
0.1 & 0.02000 & $1.2343\times10^{-10}$ & 4.000 & 2.000e-03\\
0.1 & 0.01000 & $7.7159\times10^{-12}$ & 4.000 & 1.000e-03\\
0.1 & 0.00500 & $4.8232\times10^{-13}$ & 4.000 & 5.000e-04\\
\addlinespace[2pt]
1 & 0.04000 & $2.7197\times10^{-9}$ & -- & 4.000e-02\\
1 & 0.02000 & $1.6578\times10^{-10}$ & 4.036 & 2.000e-02\\
1 & 0.01000 & $1.0314\times10^{-11}$ & 4.007 & 1.000e-02\\
1 & 0.00500 & $6.4437\times10^{-13}$ & 4.001 & 5.000e-03\\
\addlinespace[2pt]
10 & 0.04000 & $2.4115\times10^{-8}$ & -- & 4.000e-01\\
10 & 0.02000 & $3.0113\times10^{-10}$ & 6.323 & 2.000e-01\\
10 & 0.01000 & $6.9395\times10^{-12}$ & 5.439 & 1.000e-01\\
10 & 0.00500 & $3.2949\times10^{-13}$ & 4.397 & 5.000e-02\\
\bottomrule
\end{tabular}
\end{table}

\FloatBarrier
\subsection{Example 3: one-dimensional periodic Jin--Xin problem}
On $[0,2\pi]$ we take
\begin{equation}
 u(x,0)=\sin x+0.2\cos2x+0.15\sin3x.
 \label{eq:1Dic}
\end{equation}
For equilibrium data, $v(x,0)=cu(x,0)$; for unprepared data, $v(x,0)=0$.  A 64-point Fourier spectral grid is used for the convergence scan, so the semidiscrete problem is a finite collection of the modal blocks covered by \corref{cor:JX-grid}.  The exact reference is assembled mode by mode from the analytical slow/fast decomposition.  For each $h$, the maximum-in-time field error is further maximized over 25 logarithmically spaced values in $10^{-5}\le\eps\le1$.

\begin{table}[!htbp]
\centering
\scriptsize
\setlength{\tabcolsep}{3pt}
\caption{Example 3(a): complete 1D refinement history for equilibrium initial data.}
\label{tab:1D-eq}
\begin{tabular}{rrrrrrr}
\toprule
$h$ & base error & Rate & SFRC error & Rate & $\varepsilon_{\max}^{\rm base}$ & $\varepsilon_{\max}^{\rm SFRC}$\\
\midrule
0.04000 & $7.5181\times10^{-4}$ & -- & $1.0570\times10^{-6}$ & -- & 1.0e-05 & 1.0e+00\\
0.02000 & $2.0153\times10^{-4}$ & 1.899 & $6.6390\times10^{-8}$ & 3.993 & 5.1e-03 & 1.0e+00\\
0.01000 & $9.3652\times10^{-5}$ & 1.106 & $4.1597\times10^{-9}$ & 3.996 & 2.0e-03 & 1.0e+00\\
0.00500 & $4.9121\times10^{-5}$ & 0.931 & $2.6030\times10^{-10}$ & 3.998 & 1.2e-03 & 1.0e+00\\
\bottomrule
\end{tabular}
\end{table}

\begin{table}[!htbp]
\centering
\scriptsize
\setlength{\tabcolsep}{3pt}
\caption{Example 3(b): complete 1D refinement history for unprepared initial data.}
\label{tab:1D-unp}
\begin{tabular}{rrrrrrr}
\toprule
$h$ & base error & Rate & SFRC error & Rate & $\varepsilon_{\max}^{\rm base}$ & $\varepsilon_{\max}^{\rm SFRC}$\\
\midrule
0.04000 & $1.8769\times10^{1}$ & -- & $7.2488\times10^{-7}$ & -- & 1.0e-05 & 1.0e+00\\
0.02000 & $4.6870\times10^{0}$ & 2.002 & $4.5528\times10^{-8}$ & 3.993 & 1.0e-05 & 1.0e+00\\
0.01000 & $1.1684\times10^{0}$ & 2.004 & $2.8526\times10^{-9}$ & 3.996 & 1.0e-05 & 1.0e+00\\
0.00500 & $2.9035\times10^{-1}$ & 2.009 & $1.7851\times10^{-10}$ & 3.998 & 1.0e-05 & 1.0e+00\\
\bottomrule
\end{tabular}
\end{table}

\begin{figure}[!htbp]
\centering
\begin{subfigure}{.49\textwidth}
\centering\includegraphics[page=7,width=\textwidth]{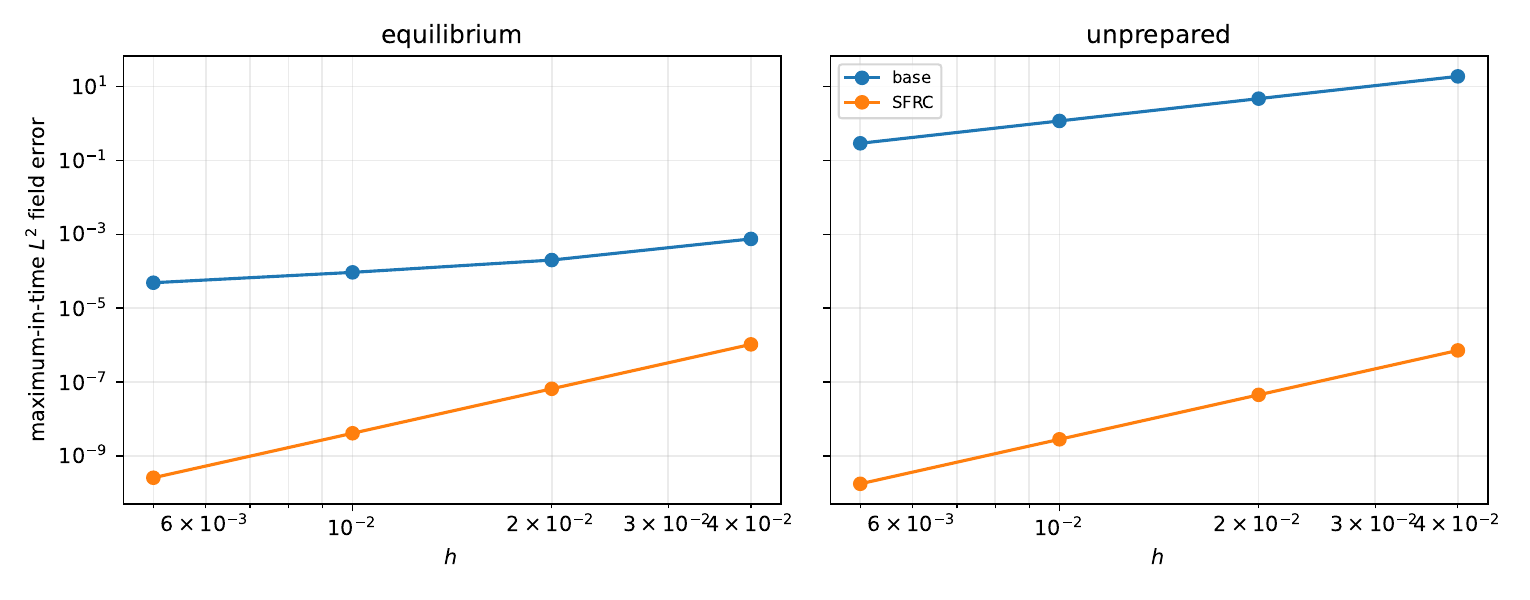}
\caption{Maximum-over-$\eps$ convergence.}
\end{subfigure}\hfill
\begin{subfigure}{.49\textwidth}
\centering\includegraphics[page=8,width=\textwidth]{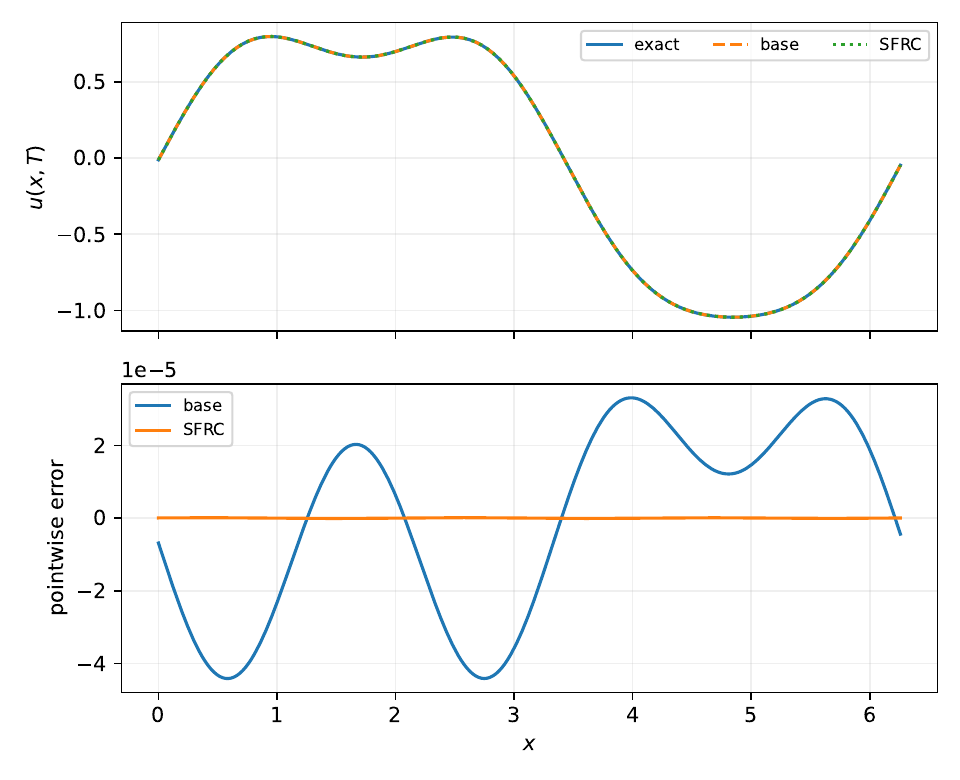}
\caption{Profile and pointwise error for $\eps=h=0.04$.}
\end{subfigure}
\caption{Example 3: one-dimensional periodic Jin--Xin problem.}
\label{fig:1D}
\end{figure}

The SFRC histories are essentially fourth order for both preparations: the final observed rates are $3.998$ in both tables.  The base method behaves differently because the maximizer moves toward small intermediate relaxation times; for equilibrium data its final rate is about one, and for unprepared data the unresolved fast layer produces much larger errors.  The right panel of Figure~\ref{fig:1D} turns the convergence statement into a physical-space comparison.  At $\eps=h=0.04$, the SFRC profile is visually indistinguishable from the exact field at plotting resolution, whereas the base pointwise error remains easily visible.  This experiment therefore checks the fixed-grid Fourier corollary simultaneously on the three active spatial frequencies in \eqref{eq:1Dic}.

\FloatBarrier
\subsection{Example 4: two-dimensional periodic relaxation system}
For \eqref{eq:2D} we choose
\[
 a_x=1.5,\quad a_y=1.3,\quad c_x=0.55,\quad c_y=0.35,
\]
and the smooth initial macro field
\begin{equation}
 u(x,y,0)=\sin x\sin y+0.25\cos(2x-y)+0.15\sin(x+2y),
 \label{eq:2Dic}
\end{equation}
with equilibrium fluxes $v=c_xu$ and $q=c_yu$.  The convergence scan uses a $16\times16$ Fourier grid and 17 logarithmically spaced relaxation parameters in $10^{-4}\le\eps\le1$; the plotted fields are regenerated on finer grids only to improve visual resolution.  The exact reference is the matrix exponential of each $3\times3$ Fourier block.  The theorem relevant here is therefore the fixed-spatial-resolution temporal result \thmref{thm:2D-UA4}, not a PDE estimate uniform as the Fourier cutoff tends to infinity.

\begin{table}[!htbp]
\centering
\scriptsize
\setlength{\tabcolsep}{3pt}
\caption{Example 4: complete two-dimensional refinement history after maximizing the field error over $10^{-4}\le\eps\le1$.  The final column is the accuracy ratio $E_{\rm base}/E_{\rm SFRC}$ at the same $h$; it is not a CPU-time or work--precision speedup.}
\label{tab:2D}
\begin{tabular}{rrrrrrr}
\toprule
$h$ & base error & Rate & SFRC error & Rate & $\varepsilon_{\max}^{\rm SFRC}$ & Error ratio\\
\midrule
0.04000 & $8.2683\times10^{-4}$ & -- & $3.2238\times10^{-7}$ & -- & 1.0e+00 & 2.56e+03\\
0.02000 & $2.0864\times10^{-4}$ & 1.987 & $2.0707\times10^{-8}$ & 3.961 & 3.2e-01 & 1.01e+04\\
0.01000 & $9.7571\times10^{-5}$ & 1.096 & $1.3177\times10^{-9}$ & 3.974 & 3.2e-01 & 7.40e+04\\
0.00500 & $4.9373\times10^{-5}$ & 0.983 & $8.3130\times10^{-11}$ & 3.986 & 3.2e-01 & 5.94e+05\\
\bottomrule
\end{tabular}
\end{table}

\begin{figure}[!htbp]
\centering
\includegraphics[page=9,width=.64\textwidth]{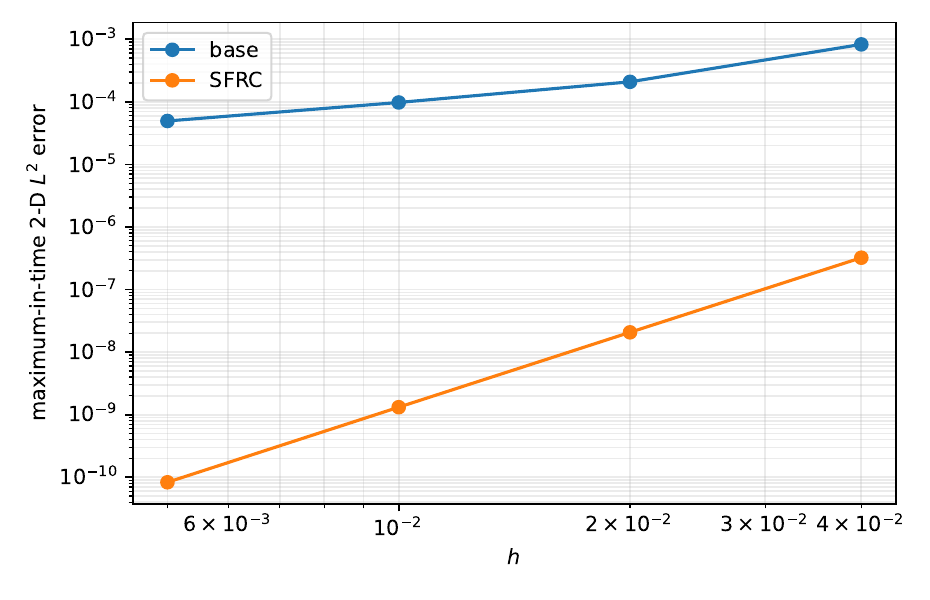}
\caption{Example 4: two-dimensional maximum-in-time error, maximized over the relaxation scan.}
\label{fig:2Dconv}
\end{figure}

The full history in Table~\ref{tab:2D} gives SFRC rates $3.961$, $3.974$, and $3.986$ on the three refinements, while the base method approaches first order after a transient.  The accuracy ratio $E_{\rm base}/E_{\rm SFRC}$ grows from about $2.6\times10^3$ at $h=0.04$ to about $5.9\times10^5$ at $h=0.005$.  This ratio compares errors only; it is not a computational-efficiency or work--precision speedup, since SFRC additionally uses the exact slow projector and exact fast semigroup.  The SFRC maximizer is at $\eps=1$ on the coarsest step and near $\eps=0.32$ on the three refined steps, so the reported fourth-order rate is not generated by the strict endpoint.

The theorem is a full-state statement and does not require equilibrium-prepared data.  To test that part of the claim in two dimensions, we repeat the same relaxation scan with the deliberately unprepared fluxes $v(x,y,0)=q(x,y,0)=0$.  Table~\ref{tab:pde2d-unprepared} reports the complete refinement history.  The SFRC rates are $3.951$, $3.973$, and $3.986$, while the base maximizer stays at the smallest sampled relaxation time $10^{-4}$ and exhibits the much larger unresolved-fast error.  This complements the unprepared one-dimensional experiment and checks the preparation-free mechanism on the full three-variable Fourier blocks.

\begin{table}[!htbp]
\centering
\scriptsize
\setlength{\tabcolsep}{3pt}
\caption{Example 4(b): two-dimensional unprepared-data refinement history with $v(x,y,0)=q(x,y,0)=0$, after maximizing the full-state field error over $10^{-4}\le\eps\le1$.}
\label{tab:pde2d-unprepared}
\begin{tabular}{rrrrrrr}
\toprule
$h$ & base error & Rate & SFRC error & Rate & $\varepsilon_{\max}^{\rm base}$ & $\varepsilon_{\max}^{\rm SFRC}$\\
\midrule
0.04000 & $1.1788\times10^{0}$ & -- & $2.7082\times10^{-7}$ & -- & 1.0e-04 & 3.2e-01\\
0.02000 & $2.9043\times10^{-1}$ & 2.021 & $1.7508\times10^{-8}$ & 3.951 & 1.0e-04 & 3.2e-01\\
0.01000 & $7.0496\times10^{-2}$ & 2.043 & $1.1146\times10^{-9}$ & 3.973 & 1.0e-04 & 3.2e-01\\
0.00500 & $1.6660\times10^{-2}$ & 2.081 & $7.0331\times10^{-11}$ & 3.986 & 1.0e-04 & 3.2e-01\\
\bottomrule
\end{tabular}
\end{table}

To make the correction visible beyond a convergence table, Figures~\ref{fig:2Dsol}--\ref{fig:2Dref} use the distinguished path $\eps=h$.  The solution fields themselves are smooth and therefore look similar, but the logarithmic error maps reveal the distributed improvement over the entire domain.

\begin{figure}[!htbp]
\centering
\includegraphics[page=10,width=.98\textwidth]{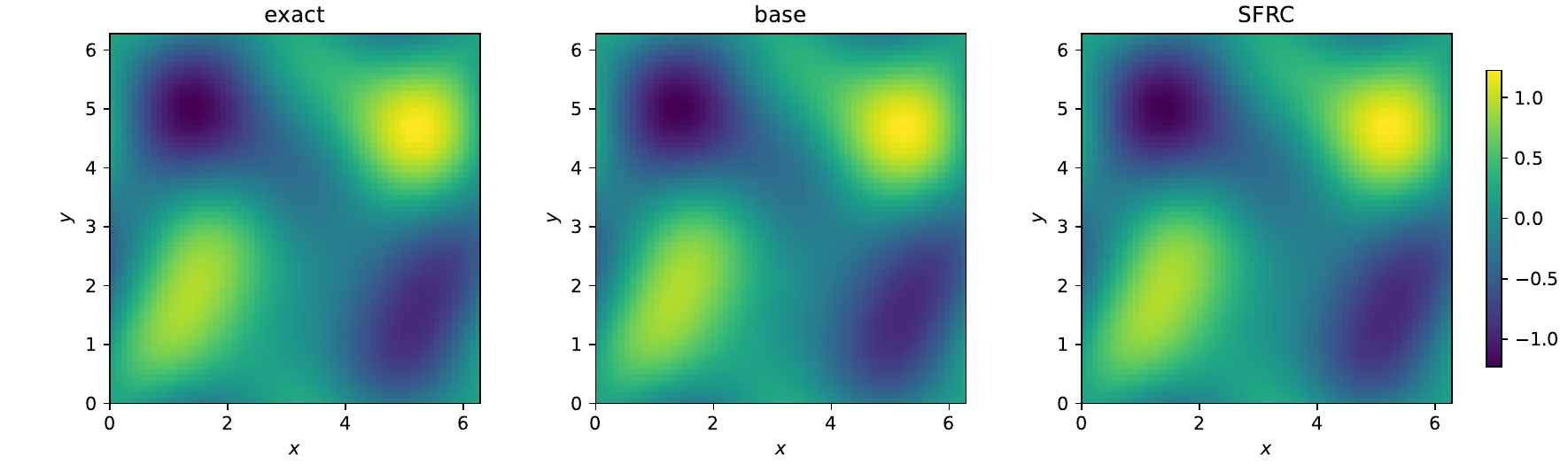}
\caption{Example 4: exact, base, and SFRC macro fields at $T=0.2$ for $\eps=h=0.04$.}
\label{fig:2Dsol}
\end{figure}

\begin{figure}[!htbp]
\centering
\includegraphics[page=11,width=.98\textwidth]{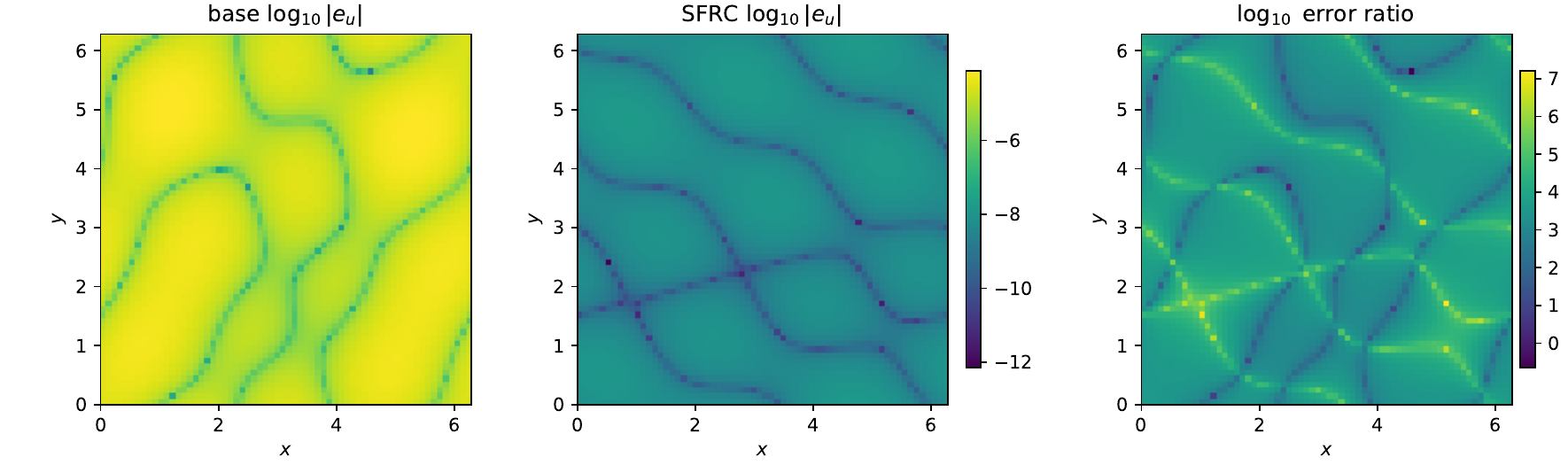}
\caption{Example 4: logarithmic macro-error fields for $\eps=h=0.04$.  The third panel shows the local quantity $\log_{10}[(|e_{\rm base}|+\tau)/(|e_{\rm SFRC}|+\tau)]$, i.e. a pointwise error ratio on a logarithmic scale.}
\label{fig:2Derr}
\end{figure}

\begin{figure}[!htbp]
\centering
\includegraphics[page=12,width=.92\textwidth]{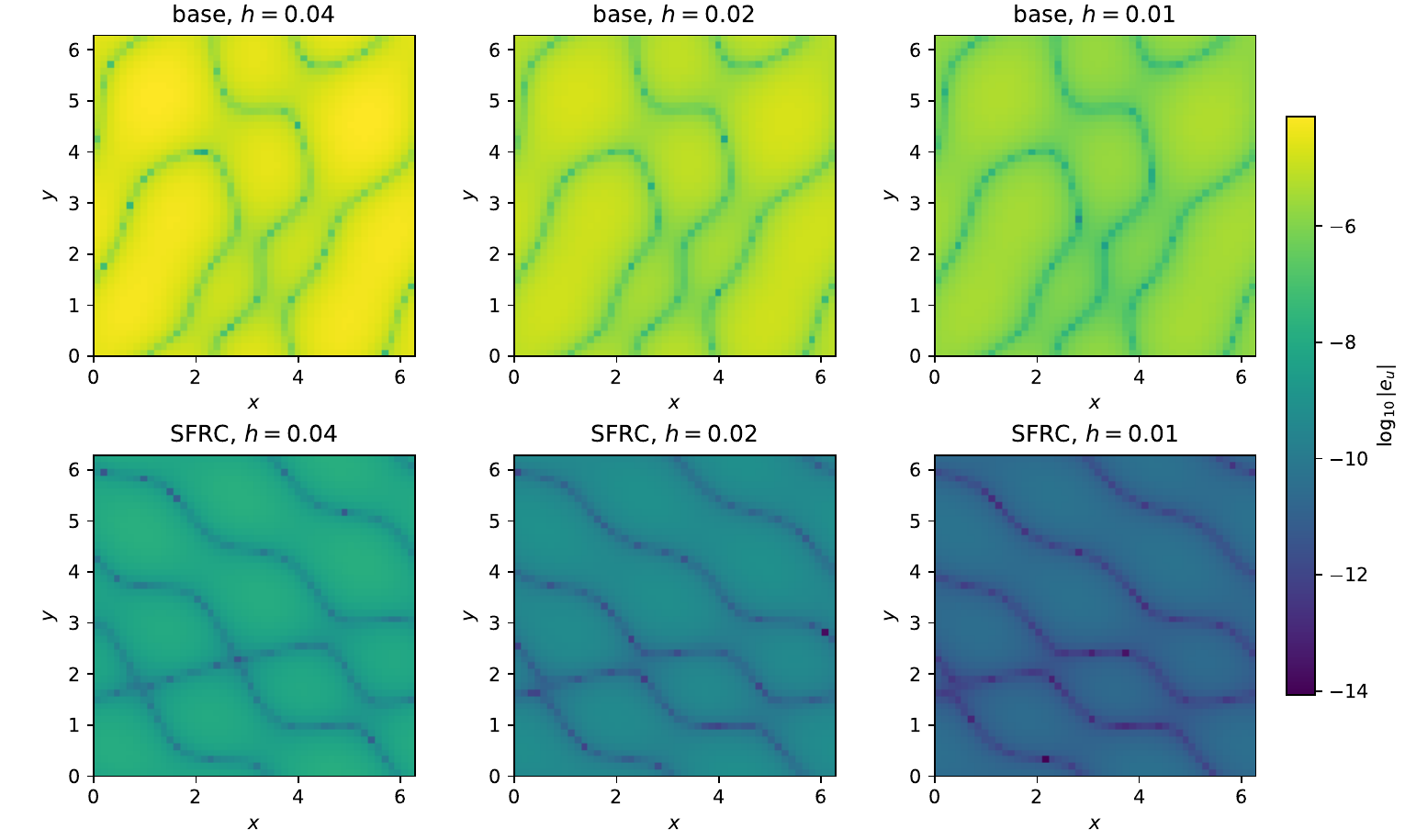}
\caption{Example 4: distinguished-scale refinement.  Base errors are shown in the top row and SFRC errors in the bottom row using one shared logarithmic scale.}
\label{fig:2Dref}
\end{figure}

The 2D experiment corroborates the modewise theorem on the chosen finite Fourier grid and confirms that the correction acts simultaneously on all active smooth modes rather than on one hand-picked scalar amplification factor.  It does not, and cannot, establish a spatially uniform PDE theorem: the analytical constant is permitted to depend on the fixed Fourier cutoff, exactly as stated in \thmref{thm:2D-UA4}.

\section{Discussion, limitations, and implications for method design}
The rigorous results establish compatibility of several properties that can appear to compete with one another.  The corrected scheme uses the same compact fourth-order two-derivative formula on the slow component, so the stage coefficients, the full-vector-field derivative evaluations, and the two implicit solves are not redesigned.  \Propref{prop:mixed-preservation} shows at the operator level that the correction is fifth order for every fixed nonzero relaxation time.  Therefore the ordinary fourth-order order conditions and the evaluated non-commuting mixed operator expansion of the original method are not merely observed numerically; they are mathematically identical through the design order.  The pure-fast response is the exact exponential and hence has the L-stability property, although it is not the original rational stability function.  The strict relaxation limit is fourth order.  Finally, \thmref{thm:JX-UA4} and \thmref{thm:2D-UA4} prove full-state uniform fourth-order time accuracy on the stated fixed Fourier grids; the constants are not claimed to be uniform as the spatial cutoff tends to infinity.

The theorem is stronger than a preparation-dependent estimate because the fast error is identically zero.  This should not be confused with a claim that arbitrary high-order IMEX methods automatically have preparation-free uniform accuracy.  Existing rigorous analyses of IMEX-BDF and IMEX-RK methods use energy estimates, structural stability, stage conditions, and high regularity or consistent initial data to control unresolved fast layers \cite{HuShuBDF2021,HuShuRK2025,MaHuang2025}.  The present method avoids that particular difficulty by paying for exact fast spectral evolution.  Its theoretical standard is therefore high, but its immediate applicability is narrower.

The most important limitation is computational structure.  An exact slow spectral projector and the exact action $e^{hL_\eps}\Pf$ are inexpensive for the $2\times2$ and $3\times3$ Fourier blocks used here and can be precomputed for constant coefficients.  They are not inexpensive or even globally meaningful for a nonlinear conservation law with a space-dependent Jacobian, for a large unstructured discretization, or for a relaxation operator whose slow subspace changes strongly in space and time.  Consequently SFRC should be viewed as a rigorous benchmark construction.  Its principal contribution is not the computational novelty of using an exact spectral projector or exact exponential; those are deliberately strong ingredients.  For the same reason, the large base/SFRC error ratios reported in the numerical section are accuracy comparisons and should not be interpreted as CPU-time or work--precision speedups.  The contribution is the structural conclusion that the original compact IMEX temporal core is not itself the obstruction to uniform fourth order, together with a proof of exactly what additional slow--fast information is sufficient in the linear setting.  In particular, \corref{cor:slow-core} shows that the slow-projected response of the unchanged base core already has a uniform $O(h^5)$ local defect; the correction is needed to identify and preserve that finite-$\eps$ slow response and to control the complementary fast component, not to repair a hidden low-order slow-stage formula.

This viewpoint suggests a practical next step.  Replace $\Ps$ by a local slow-manifold approximation, for example a Chapman--Enskog or micro--macro projector through enough orders in $\eps$, and replace the exact fast semigroup by a source exponential, a rational approximation, or a Krylov action that is exact to the required uniform order.  The present proofs show what must be preserved.  The slow-block one-step defect must remain $O(h^5)$ uniformly in $\eps/h$; the approximate fast block must not reintroduce an $O(1)$ initial-layer numerical error; and the overall correction must remain $O(h^5)$ for fixed $\eps$ so the base method's mixed expansion through fourth order is not altered.  These are sharper design conditions than simply requiring the correct two endpoint limits.

The natural question is whether the same conclusion can be promoted to a genuinely nonlinear relaxation theorem.  The present paper does not make that claim, and the linear proof cannot simply be copied: for nonlinear problems the slow projector depends on the state, need not commute with the flow, and an exact fast semigroup is generally unavailable.  Nevertheless, the analysis identifies a concrete sufficient-condition program.  A nonlinear extension would require a smooth $\eps$-dependent slow manifold with a uniformly Lipschitz local retraction, a corrected slow step whose local defect is $O(h^5)$ uniformly in $\eps/h$, a realizable fast evolution that is uniformly stable and fourth-order accurate through the initial layer, and a fixed-$\eps$ correction that remains $O(h^5)$ so the base mixed expansion is unchanged.  Under such ingredients, the stability-plus-telescoping argument of Section~4 has a nonlinear analogue through a discrete Gronwall estimate.  What is currently missing is not the global error mechanism but a proof that practical Chapman--Enskog or micro--macro approximations satisfy those uniform local and fast-layer bounds.  Establishing that point is a separate nonlinear analysis and is therefore left outside the claims of this paper.

A second implication concerns the earlier equilibrium-defect correction from the companion analysis \cite{HuoAA2026}.  That construction remains useful as a minimal strict-limit repair because it needs only equilibrium information and no spectral decomposition.  The current analysis does not invalidate that role.  Rather, it clarifies the hierarchy: an endpoint correction can restore strict AA4 at essentially no stage cost, while a uniform correction must control both the slow response and the fast spectral component throughout the intermediate scale.  SFRC supplies one rigorous way to do so in the linear setting.

Finally, the exact factorization in \lemref{lem:factorization} is useful independently of the specific corrected method.  It exposes the mechanism by which the apparent singular coefficients of the base step cancel on the finite-$\eps$ slow eigenspace.  The factor $cd+\lambda_s=O(\eps)$ is precisely what compensates the implicit denominators when $\eps/h$ is small, while the $x^4$ growth of the denominator controls the opposite fixed-$\eps$ regime.  This is a concrete algebraic explanation of why a slow-manifold-aware analysis can recover high order even when the limiting-equilibrium map of the unprojected base scheme is only first order.

\section{Conclusions}
We have constructed a slow--fast response correction of a compact two-stage fourth-order two-derivative IMEX method under the explicit requirement that the principal structural properties of the original temporal discretization be retained.  The corrected method applies the unchanged compact IMEX core to the exact slow spectral component and advances the fast complement with its exact semigroup.  It introduces no additional implicit solve, although it requires the additional exact spectral operations stated explicitly in Section~3.  For fixed nonzero relaxation time the correction is a fifth-order perturbation, so the original fourth-order non-commuting mixed expansion is preserved exactly.  The pure-fast response is $e^z$ and is therefore L-stable; this preserves the stability property while replacing the base rational response by exact exponential damping.

The central theoretical contribution is a complete uniform-accuracy proof rather than a short endpoint argument.  An abstract slow--fast decomposition reduces the full-state error to the slow block and converts a uniform fifth-order local defect into a fourth-order global estimate.  For the linear Jin--Xin Fourier block, the local defect is then proved uniformly: the projected base amplification admits an exact $h^5$ factorization, the distinguished variable $x=\eps/h$ is introduced, the two implicit denominators are bounded from below for Fourier modes, and the remaining coefficient is shown to stay bounded for all $x\ge0$.  This verifies the hypothesis of the abstract theorem rather than assuming it.  A two-dimensional three-variable relaxation system reduces to the same coupled slow block, giving a modewise extension on fixed Fourier grids.

The main conclusion is therefore structural.  The order loss of the unmodified compact scheme is not evidence that its two-stage/two-solve fourth-order core is incompatible with uniform accuracy.  Once the finite-$\eps$ slow response is isolated and the fast component is controlled exactly, the original slow-stage core can coexist with its fourth-order mixed expansion, the L-stability property, strict-limit fourth order, and full-state uniform-in-$\eps$ fourth-order time accuracy.  The principal contribution is this compatibility theorem and the explicit Jin--Xin proof of the uniform slow defect, not the use of an exact projector by itself.  The uniform slow-block corollary makes the mechanism especially transparent: after restriction to the exact finite-$\eps$ slow eigenspace, the unchanged compact base core already has a uniform fifth-order one-step defect, so the earlier strict-equilibrium order loss cannot be attributed to an intrinsic failure of its slow-stage fourth-order dynamics.  The proof is rigorous for the stated linear spectral setting and fixed Fourier resolutions; the numerical experiments corroborate, but do not prove, its quantitative consequences.  A general nonlinear UA4 theorem remains open because state-dependent slow manifolds and approximate fast solvers require additional uniform local and initial-layer estimates.  The unresolved practical problem is therefore to replace the exact projector and exact fast semigroup by local Chapman--Enskog, micro--macro, source-exponential, rational, or Krylov approximations while preserving the same uniform $O(h^5)$ slow defect and sufficiently accurate fast-layer evolution.

\section*{Data and code availability}

\section*{Acknowledgments}
The author thanks colleagues for helpful discussions that sharpened the distinction between strict-limit asymptotic accuracy and uniform accuracy.  Zhixin Huo gratefully acknowledges support from the Key Program of Henan Higher Education Institutions (Grant No.~26A110007), the Young Talents Fund of Henan Province (Grant No.~252300423500), and the Doctoral Startup Foundation of Henan Polytechnic University (Grant No.~B2024-60).

\appendix
\section{Additional algebra behind the factorization}
This appendix records intermediate identities that make \lemref{lem:factorization} independently checkable from the displayed algebra.  Because $B^2=-\eps^{-1}B$ for \eqref{eq:JXAB}, the inverse of $I+\alpha B$ is
\begin{equation}
 (I+\alpha B)^{-1}
 =I-\frac{\alpha}{1-\alpha/\eps}B,
 \label{eq:Binv}
\end{equation}
whenever the denominator is nonzero.  On the slow eigenvector, $Lr_s=\lambda r_s$, so the midpoint equation can be written
\[
 \left(I-\frac h2B+\frac{h^2\lambda}{8}B\right)U^*
 =\left(I+\frac h2A+\frac{h^2\lambda}{8}A\right)r_s.
\]
Using $Ar_s=\lambda r_s-Br_s$ reduces the right-hand side to a combination of $r_s$ and $Br_s$.  The source action is finite on the slow eigenvector because
\begin{equation}
 Br_s
 =\begin{pmatrix}0\\(cd+\lambda)/(\eps d)\end{pmatrix}
 =\begin{pmatrix}0\\(a^2d^2-\lambda^2)/d\end{pmatrix},
 \label{eq:Brfinite}
\end{equation}
where \eqref{eq:char} was used.  This identity is the first cancellation of the singular $1/\eps$ factor.

For completeness, let $V=(r_s,Br_s)$.  Direct application of $A$ and $B$ to these two basis vectors gives
\[
 A r_s=\lambda r_s-Br_s,\qquad B r_s=Br_s,
\]
\[
 A(Br_s)=-\frac{cd+\lambda}{\eps}r_s-\lambda Br_s,\qquad
 B(Br_s)=-\frac1\eps Br_s,
\]
which proves \eqref{eq:reduced-basis}.  In the same basis the two implicit matrices are explicitly
\begin{align}
 \widehat M_*&=I-\frac h2\widehat B+\frac{h^2}{8}\widehat B\widehat L
 =\begin{pmatrix}
 1&0\\ \dfrac{h(h\lambda-4)}8&1+\dfrac{h}{2\eps}+\dfrac{h^2}{8\eps^2}-\dfrac{cdh^2}{8\eps}
 \end{pmatrix},\label{eq:Mstar-hat}\\
 \widehat M_{n+1}&=I-h\widehat B+\frac{h^2}{6}\widehat B\widehat L
 =\begin{pmatrix}
 1&0\\ \dfrac{h(h\lambda-6)}6&1+\dfrac{h}{\eps}+\dfrac{h^2}{6\eps^2}-\dfrac{cdh^2}{6\eps}
 \end{pmatrix}.\label{eq:Mend-hat}
\end{align}
Their determinants are therefore exactly the quantities stated in \eqref{eq:det-certificate}.  Also
\[
 \frac{\ell_s^TBr_s}{\ell_s^Tr_s}=-\frac{cd+\lambda}{1+2\eps\lambda},
\]
which proves \eqref{eq:omega}.  Hence every ingredient of the scalar rational function $q=\omega^T\widehat S e_1$ is explicit before any polynomial expansion is performed.

The same reduction applies to the final implicit solve.  After expressing every occurrence of $a^2d^2$ through \eqref{eq:achar}, all remaining denominators can be collected into the two scalar factors $D_8$ and $D_6$ of \eqref{eq:D8}--\eqref{eq:D6}.  The scalar left projection contributes the factor
\begin{equation}
 \ell_s^Tr_s=-\frac{1+2\eps\lambda}{\eps d}.
 \end{equation}
Since $1+2\eps\lambda=\zeta$, the spectral non-coalescence condition is exactly the condition that this biorthogonal normalization not degenerate after the natural $1/\eps$ scaling is removed.

Clearing the denominator $24(1+2\eps\lambda)D_8D_6$ from $q-\Tfour(h\lambda)$ gives the certificate \eqref{eq:certificate}.  Multiplying out the two triangular inverses \eqref{eq:Mstar-hat}--\eqref{eq:Mend-hat} and using \eqref{eq:achar} cancels every monomial of total $h$-degree below five; the remaining numerator is exactly $-h^5\lambda^3(cd+\lambda)N$.  Thus the certificate vanishes identically and \eqref{eq:factor} follows.  The complete polynomial $N$ is given in \eqref{eq:N}; no omitted higher-order term is present.

\section{Why the global proof has no stiff exponential constant}
A common failure mode in stiff error estimates is to derive a local truncation error and then propagate it using a stability constant that grows like $e^{T/\eps}$.  That does not occur here for two separate reasons.  First, the fast part is exact and cancels from the numerical error as shown in \eqref{eq:fastcancel}; no bound on powers of an approximate fast amplification matrix is required.  Second, the only propagated numerical block is the slow block, whose generator is uniformly bounded by \eqref{eq:slowbound}.  The one-step bound in \lemref{lem:slow-stability} is therefore $e^{C h}$, not $e^{Ch/\eps}$.  Raising it over $T/h$ steps gives $e^{CT}$ with a constant independent of $\eps$.

For the Jin--Xin mode, this can be seen directly from the scalar expression.  The exact slow eigenvalue remains bounded as $\eps\to0$, while the corrected fast eigenvalue is represented by $e^{h\lambda_f}$ exactly.  Hence the only accumulated defect is
\[
 q^n-e^{nh\lambda_s},
\]
and the factor multiplying the local $O(h^5)$ defect is bounded by $e^{CT}$.  This is the precise stability step needed to prevent a stiff exponential constant from entering the global estimate.

\section{Implementation remarks}
The one-dimensional implementation does not obtain the slow projector from a numerically ill-conditioned eigendecomposition when $\eps$ is tiny.  It uses the scaled closed form \eqref{eq:PsJX} and the cancellation-free root \eqref{eq:lams}.  Likewise, the projected scalar base response is evaluated through the exact factorization of \lemref{lem:factorization}.  These implementation choices matter: direct double-precision eigendecomposition of matrices containing entries of size $1/\eps$ can obscure the uniform limit through roundoff even when the mathematical projector is perfectly bounded.

{\scriptsize
\sloppy
\setlength{\columnsep}{18pt}
\begin{multicols}{2}

\end{multicols}
}
\end{document}